\documentclass[12 pt]{amsart}
\usepackage{amsmath,amsfonts,amsthm,mathrsfs,amssymb,bbm}
\usepackage{tikz}
\usepackage{tikz-cd}
\usepackage[margin=1.2in]{geometry}
\usepackage{appendix}

\newtheorem{thm}{Theorem}
\newtheorem{lem}{Lemma}[section]
\newtheorem{prop}{Proposition}[section]
\newtheorem{defn}{Definition}[section]
\newtheorem{cor}{Corollary}[section]

\usepackage{chngcntr}
\counterwithin{equation}{section}

\newcommand{\bp}{\begin{proof}}
\newcommand{\ep}{\end{proof}}
\newcommand{\be}{\begin{enumerate}}
\newcommand{\ee}{\end{enumerate}}

\newcommand{\zz}{\mathbb{Z}} 
\newcommand{\rr}{\mathbb{R}} 
\newcommand{\nn}{\mathbb{N}} 

\usepackage{mathtools} 

\newcommand{\dist}{\mathrm{dist}}
\newcommand{\Ker}{\mathrm{ker}}
\newcommand{\mx}[1]{\begin{smallmatrix} #1 \end{smallmatrix}} 
\newcommand{\supp}{\mathrm{supp}}
\newcommand{\diam}{\mathrm{diam}}
\DeclareMathOperator{\sgn}{sgn}
\newcommand{\C}{\tilde{C}}

\newcommand{\A}{\mathcal{A}} 

\newcommand{\M}{\mathcal{M}}
\newcommand{\Z}{\mathcal{Z}}
\newcommand{\1}{\mathbbm{1}}
\newcommand{\B}{\mathcal{E}}
\newcommand{\chimn}{\zeta}

\DeclareFontFamily{U}{mathx}{\hyphenchar\font45}
\DeclareFontShape{U}{mathx}{m}{n}{
      <5> <6> <7> <8> <9> <10>
      <10.95> <12> <14.4> <17.28> <20.74> <24.88>
      mathx10
      }{}
\DeclareSymbolFont{mathx}{U}{mathx}{m}{n}
\DeclareFontSubstitution{U}{mathx}{m}{n}
\DeclareMathAccent{\widecheck}{0}{mathx}{"71}

\usepackage{scalerel,stackengine}
\stackMath
\newcommand\reallywidehat[1]{%
\savestack{\tmpbox}{\stretchto{%
  \scaleto{%
    \scalerel*[\widthof{\ensuremath{#1}}]{\kern-.6pt\bigwedge\kern-.6pt}%
    {\rule[-\textheight/2]{1ex}{\textheight}}
  }{\textheight}%
}{0.5ex}}%
\stackon[1pt]{#1}{\tmpbox}%
}
\begin{document}

\title[Averaging operators on the Heisenberg group]{$L^p$ regularity for a class of averaging operators on the Heisenberg group}
\author[G. Bentsen]{Geoffrey Bentsen}
\address{Geoffrey Bentsen \\ Department of Mathematics \\ University of Wisconsin-Madison \\ WI 53706, USA}
\email{gbentsen@wisc.edu}
\thanks{Research supported in part by NSF grant DMS 1764295}
\date{\today}

\begin{abstract}
    We prove $L^p_{comp}(\rr^3)\to L^p_{s}(\rr^3)$ boundedness for averaging operators associated to a class of curves in the Heisenberg group $\mathbb{H}^1$ via $L^2$ estimates for related oscillatory integrals and Bourgain-Demeter decoupling inequalities on the cone. We also construct a Sobolev space adapted to translations on the Heisenberg group to which these averaging operators map all $L^p$ functions boundedly. 
\end{abstract}
\maketitle

\section{Introduction}

Let $\mathbb{H}^1$ be the Heisenberg group, that is $\rr^3$ with the product
\begin{equation*}
    (x_1,x_2,x_3)\odot (w_1,w_2,w_3)=(x_1+w_1,x_2+w_2,x_3+w_3+\tfrac12(x_1w_2-x_2w_1)).
\end{equation*}
Let $\gamma:[0,1]\to \rr^3$ be a regular smooth curve (e.g. $C^\infty$ and $\gamma'\ne 0$) whose tangent vector is nowhere parallel to $(0,0,1)$. Then without loss of generality (though possibly with a reordering of the last two coordinates) we can write $\gamma(t)=(t,\gamma_2(t),\gamma_3(t))$, where $\gamma_2,\gamma_3\in C^\infty(\rr)$. Let $\mu$ be a smooth measure supported on $\gamma([0,1])$ and for $f\in \mathcal{S}(\rr^3)$ define the generalized averaging operator
\begin{align*}
    A f(x)&=\int f(\gamma(t)^{-1}\odot x)\, d\mu(t).
\end{align*}

We are interested in finding the regularity properties of $A$ in terms of Sobolev spaces. For Euclidean averaging operators over curves, $L^p$-Sobolev bounds are closely related to the curvature and torsion properties of $\gamma$; in particular if $\gamma'(t),\gamma''(t), \gamma'''(t)$ are linearly independent for $t\in[0,1]$, or equivalently
\[
    \det\left(\mx{\gamma_2''(t) & \gamma_3''(t) \\ \gamma_2'''(t) & \gamma_3'''(t)}\right)\ne 0,
\]
then the Euclidean averaging operator over the curve $\gamma$ is bounded from $L^p(\rr^3)$ to $L^p_{1/p}(\rr^3)$ for $p>4$, see \cite{PrSe07}. Here and throughout the paper, $L^p_s(\rr^3)$ denotes the standard Sobolev space on $\rr^3$ with respect to Lebesgue measure. A similar curvature and torsion condition is found in \cite{sec99}, where Secco investigated $L^p$-improving estimates for averaging operators over curves in $\mathbb{H}^1$ using the moment curve $\gamma(t)=(t,t^2,\alpha t^3)$ as a model case. In her analysis the best possible $L^p(\rr^3)\to L^q(\rr^3)$ bounds for $A$ occur when 
\begin{equation}\label{piLfold}
    \det\left(\mx{\gamma_2''(t) & \gamma_3''(t) \\ \gamma_2'''(t) & \gamma_3'''(t)}\right)+\tfrac12(\gamma_2''(t))^2\ne0,
\end{equation}
and
\begin{equation}\label{piRfold}
    \det\left(\mx{\gamma_2''(t) & \gamma_3''(t) \\ \gamma_2'''(t) & \gamma_3'''(t)}\right)-\tfrac12(\gamma_2''(t))^2\ne 0.
\end{equation}
In the case of the moment curve these conditions imply $\alpha\ne \pm1/6$. Secco showed that these conditions are equivalent to a group-invariant version of nondegenerate curvature and torsion adapted to right and left translations on the Heisenberg group. The work of Pramanik and Seeger (see \cite{PrSe19}) shows that if one assumes $\gamma$ satisfies \eqref{piLfold} and \eqref{piRfold}, the operator $A$ also maps boundedly from $L^p_{comp}(\rr^3)$ to $L^p_{1/p}(\rr^3)$ for $p>4$. In this paper we assume \eqref{piLfold}, but relax the assumption \eqref{piRfold} and instead consider the extreme case where \eqref{piRfold} does not hold anywhere. Equivalently, we consider the following.

\begin{thm}\label{mainthm}
    Suppose that $t\mapsto \gamma(t)=(t,\gamma_2(t),\gamma_3(t))$ is a curve such that $\gamma_2''(t)\ne 0$ and 
    \[
    \det\left(\mx{\gamma_2''(t) & \gamma_3''(t) \\ \gamma_2'''(t) & \gamma_3'''(t)}\right)=\frac12 (\gamma_2''(t))^2
    \]
    for all $t\in I$. Then $A$ maps boundedly from $L^p_{comp}(\rr^3)$ to $L^p_{1/p}(\rr^3)$ for $p>4$. 
\end{thm}
Thus the case where \eqref{piRfold} does not hold in an interval gives the same $L^p$ regularity as the case where \eqref{piRfold} always holds. This seems to suggest that the condition \eqref{piLfold} should be sufficient to prove $L^p_{comp}(\rr^3)\to L^p_{1/p}(\rr^3)$ boundedness for $A$. However, it is unknown whether these estimates hold in intermediate cases, where \eqref{piRfold} does not hold for an isolated point $t_0$. 

The averaging operators studied in \cite{PrSe07} are bounded on $L^p_{1/p}(\rr^3)$ for all $f\in L^p(\rr^3)$, not just those $f$ with compact support. This is because the Euclidean averaging operators are of convolution type, hence they commute with translation on $\rr^3$. If $f$ is compactly supported on $B_1(0)$ then $A f$ lies in a fixed dilate of the support of $f$; thus if $A $ were to commute with Euclidean translations we could drop the assumption that $f$ is compactly supported by splitting $f$ into compactly supported pieces, using almost disjoint support and the fact that Fourier multipliers (in particular Bessel potentials) also commute with (Euclidean) translation. Since $A$ instead commutes with Heisenberg translations, we cannot use this argument to prove boundedness on $L^p_{1/p}(\rr^3)$ for non-compact $f\in L^p(\rr^3)$. However, we can prove that $A$ is bounded from $L^p(\rr^3)$ to an analogue of the space $L^p_{1/p}(\rr^3)$ adapted to translations on the Heisenberg group which we now introduce.

Define $\Lambda:=\{(x_1,x_2,x_3+1/2x_1x_2) \, : \, x_j\in\zz\}\subset\mathbb{H}^1$. Let $\mathcal{R}_{\lambda}$ denote right (Heisenberg) translation by $\lambda\in \Lambda$. It is easy to see that $\Lambda$ is a uniform lattice on $\mathbb{H}^1$ whose action on $\mathbb{H}^1$ is thus free and properly discontinuous (see \cite{Mo15}, Chapter 4). Thus we can pick $\psi\in C_c^\infty(B_2(0))$ such that $0\le \psi\le 1$ and $\sum_{\lambda\in \Lambda}\psi_{\lambda}\equiv 1$, where $\psi_\lambda(x)=\psi(x\odot\lambda^{-1})=\mathcal{R}_\lambda \psi$, with finitely overlapping support. Given this partition of unity, we define the following norm.
\begin{defn}\label{heisensobonorm}
    Let $D^s$ be the Bessel potential associated to the Sobolev norm, given by
    \[
        \widehat{D^s f}(\xi):=(1+|\xi|^2)^{s/2}\hat{f}(\xi).
    \]
    Then define the space $L^p_s(\mathbb{H}^1)$ to be functions in $L^p(\rr^3)$ such that the norm
    \[
        \|f\|_{L^{p}_{s}(\mathbb{H}^1)}:=\Big\|\sum_{\lambda\in \Lambda} \mathcal{R}_\lambda D^s \psi_0 \mathcal{R}_{\lambda^{-1}} f \Big\|_{L^p(\rr^3)}
    \]
    is finite. 
\end{defn}
Theorem \ref{mainthm} then implies the following.
\begin{thm}\label{heisensobocor}
    If $\gamma$ satisfies the same conditions as in Theorem \ref{mainthm}, then $A$ is bounded from $L^p(\rr^3)$ to $L^{p}_{1/p}(\mathbb{H}^1)$ for $p>4$.
\end{thm}
It is useful to note that for functions supported in a compact set $K$, 
\[
\big\|f\big\|_{L^p_{s}(\rr^3)}\simeq_K \big\|f\big\|_{L^p_{s}(\mathbb{H}^1)}.
\]

\subsection{Background}\label{background}

The operator $A$ is an example of a Fourier Integral Operator (FIO), more specifically a generalized Radon transform over a family of curves in $\rr^3$ parametrized by $x\in\rr^3$, given by $M_x=\{(-\gamma(t))\odot x \, : \, t\in I\}$. These $M_x$ are sections of a manifold $M\subset \rr^3_x\times \rr^3_y$ defined by 
\begin{equation}\label{ManifoldM}
    M=\{(x,y) \, : \, \Phi(x,y)=0\},
\end{equation}
where $\Phi=(\Phi^2,\Phi^3)$ is defined by
\begin{align}
    \label{Phi}\Phi^2(x,y)&=x_2-y_2-\gamma_2(x_1-y_1) \\ \notag \Phi^3(x,y)&=x_3-y_3-\gamma_3(x_1-y_1)+\tfrac12(x_1\gamma_2(x_1-y_1)- x_2(x_1-y_1)).
\end{align}
The $L^2$ regularity of a generalized Radon transform is related to the geometry of the (twisted) conormal bundle $(N^*M)'$ of the incidence manifold $M$, specifically the geometry of the projections $\pi_L:(N^*M)'\to T^*\rr^3_x$ and $\pi_R:(N^*M)'\to T^*\rr^3_y$. We will introduce these objects in more detail in Section \ref{conormalbundle}.
The twisted conormal bundle $(N^*M)'$ is also called the canonical relation associated to $A$ in the more general theory for FIOs, see \cite{GrSe02,Ho71}. The case where $\pi_L$ and $\pi_R$ are local diffeomorphisms is discussed by H\"ormander in \cite{Ho71}. However, if the ambient dimension is at least 3 then it is impossible for $\pi_L$ and $\pi_R$ to be nonsingular everywhere for generalized Radon transforms over curves \cite{GrSe02}. The next best case occurs when $\pi_L$ and $\pi_R$ have fold singularities (this situation is called a two-sided fold), for which numerous authors have proven $L^2$ regularity, $L^2\to L^p$, and $L^2_\alpha\to L^q_\beta$ estimates, starting with the work of Melrose and Taylor in \cite{MeTa85}. See \cite{GrSe02,Ph95}, and \cite{PhSt91} for a survey of results and methods. 

In \cite{PrSe07}, Pramanik and Seeger were able to use the decoupling inequalities of Wolff \cite{Wo00} (and later Bourgain-Demeter \cite{BoDe15}) to bootstrap $L^2$-Sobolev regularity results of a family of averaging operators over curves into $L^p$-Sobolev estimates. These averaging operators are generalized Radon transforms associated to a two-sided fold. Pramanik and Seeger, in \cite{PrSe19}, have continued their work proving these $L^p$-Sobolev bounds for a more general class of Fourier integral operators associated to two-sided folds, incidentally providing an answer to the $L^p$ regularity of $A$ when both \eqref{piLfold} and \eqref{piRfold} hold, as mentioned above. 

A natural question to ask is whether the two-sided fold assumption in \cite{PrSe19} can be weakened. By symmetry, this is equivalent to asking whether one can drop the fold assumption on $\pi_R$, while keeping the fold assumption on $\pi_L$. Optimal regularity estimates on $L^2$ have been found for FIOs with a finite type condition on $\pi_R$ (see \cite{Co99}), and $L^2_\alpha\to L^q_\beta$ estimates have been found dropping any assumption on $\pi_R$ (see \cite{GrSe94}), but the question of $L^p$ regularity remains largely unanswered. A worst case scenario occurs when $\pi_R$ is maximally degenerate, a blowdown; examples of such operators appear in \cite{GrUh89} and \cite{GrUh90B}. 

In \cite{PrSe06}, Pramanik and Seeger were able to prove the same $L^p$ regularity estimates as the two-sided fold case (for $p>4$) hold for adjoints of a particular class of restricted $X$-ray transforms, which are examples of Radon transforms associated to a fold and a blowdown. Theorem \ref{mainthm} of this paper provides another example with a positive answer to this question, as $A$ is also an example of a generalized Radon transform associated to a fold and a blowdown; we will see why in Section 2. Surprisingly, as will be shown in Subsection \ref{modelcase}, these two operators are very closely related for certain choices of $\gamma(t)$. 

In addition, we believe that the techniques of this paper, using analogues of the family of changes of variables in \cite{PrSe19}, can generalize the $L^p$-Sobolev regularity estimates of Theorem \ref{mainthm} to a class of averaging operators over curves associated to a fold blowdown singularity, analogous to the class analyzed in \cite{PrSe19}. We are currently working on this generalization and hope to publish it in the future.

Work on $L^2$-Sobolev bounds for FIOs with one-sided fold singularities show that $L^2_{comp}(\rr^3)\to L^2_{1/4,loc}(\rr^3)$ bounds are the best possible if $\pi_R$ is a blowdown, see \cite{GrSe94}. We can thus interpolate the results of Theorem \ref{mainthm} with this $L^2$-Sobolev estimate to obtain $L^p$-Sobolev estimates for $2\le p\le 4$. To obtain estimates for $p<2$, we use analytic interpolation with the above estimate on $L^2$ and a certain Hardy space estimate. We combine the estimates for all $p$ together in the following corollary.

\begin{cor}\label{corollary}
    If $\gamma$ satisfies the conditions of Theorem \ref{mainthm} then $A$ maps boundedly from $L^p_{comp}(\rr^3)$ into $L^p_s(\rr^3)$, where $s$ lies inside the trapezoidal region illustrated below.
\end{cor}

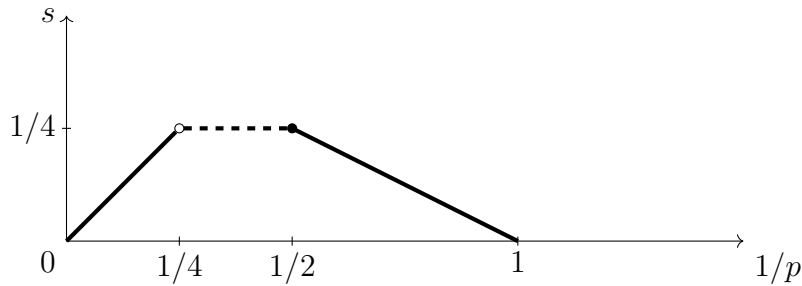
\begin{center}
    \begin{tikzpicture}[scale=6]
        \draw [<->] (0,1/2) -- (0,0) -- (3/2,0);
        \draw [ultra thick] (0,0) -- (0.242,0.242);
        \draw [ultra thick, dashed] (0.26,0.25) -- (1/2,1/4);
        \draw [ultra thick] (1/2,1/4) -- (1,0);
        \draw [thin] (1/4,.01) -- (1/4,-.01);
        \node [below] at (1/4,0) {$1/4$};
        \draw [thin] (1/2,-.01) -- (1/2,.01);
        \node [below] at (1/2,0) {$1/2$};
        \draw [thin] (1,-.01) -- (1,.01);
        \node [below] at (1,0) {$1$};
        \draw [thin] (-.01,1/4) -- (.01, 1/4);
        \node [left] at (0,1/4) {$1/4$};
        \draw (1/4,1/4) circle [radius=0.01];
        \draw [fill] (1/2,1/4) circle [radius=0.01];
        \node [below right] at (1.5,0) {$1/p$};
        \node [left] at (0,1/2) {$s$};
        \node [below left] at (0,0) {0};
    \end{tikzpicture}
\end{center}

Theorem \ref{mainthm} establishes the leftmost line segment above. We will prove the rightmost line segment in Section \ref{PRSSection}. While this corollary establishes $L^p$ regularity under the conditions of Theorem \ref{mainthm}, much less is known for more general $\gamma$. 

The recent results of \cite{PrSe19} apply to $A$ and $A^*$ for curves satisfying both \eqref{piLfold} and \eqref{piRfold}. Additionally, we can interpolate with the $L^2$ estimate found in \cite{Co99} to show that $A$ is bounded from $L_{comp}^p(\rr^3)\to L^p_{s}(\rr^3)$ for $s$ within the pentagonal region below.
        
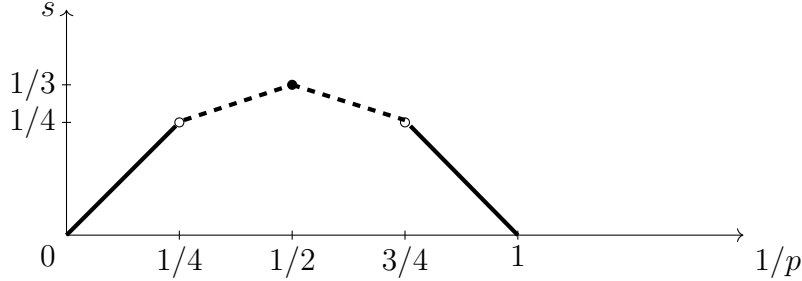
\begin{center}
         \begin{tikzpicture}[scale=6]
        \draw [<->] (0,1/2) -- (0,0) -- (3/2,0);
        \draw [ultra thick] (0,0) -- (0.242,0.242);
        \draw [ultra thick, dashed] (0.26,0.255) -- (1/2,1/3);
        \draw [ultra thick,dashed] (1/2,1/3) -- (3/4,0.255);
        \draw [ultra thick] (.762,0.242) -- (1,0);
        \draw [thin] (1/4,.01) -- (1/4,-.01);
        \node [below] at (1/4,0) {$1/4$};
        \draw [thin] (1/2,-.01) -- (1/2,.01);
        \node [below] at (1/2,0) {$1/2$};
        \draw [thin] (3/4,-0.01) -- (3/4,.01);
        \node [below] at (3/4,0) {$3/4$};
        \draw [thin] (1,-.01) -- (1,.01);
        \node [below] at (1,0) {$1$};
        \draw [thin] (-.01,1/4) -- (.01, 1/4);
        \node [left] at (0,1/4) {$1/4$};
        \draw [thin] (-.01,1/3) -- (.01,1/3);
        \node [left] at (0,1/3) {$1/3$};
        \draw (1/4,1/4) circle [radius=0.01];
        \draw [fill] (1/2,1/3) circle [radius=0.01];
        \draw (3/4,1/4) circle [radius=0.01];
        \node [below right] at (1.5,0) {$1/p$};
        \node [left] at (0,1/2) {$s$};
        \node [below left] at (0,0) {$0$};
    \end{tikzpicture}
\end{center}
As discussed in \cite{PrSe19}, the estimates for the two-sided fold case are sharp up to the boundaries of the pentagonal region. One can also see that for $p>4$ the regularity is the same as in Theorem \ref{mainthm}. We conjecture that Theorem \ref{mainthm} holds if \eqref{piLfold} holds for all $t\in I$, with no additional assumption. 

Between the assumptions of a fold-blowdown and a two-sided fold there is a large collection of finite type conditions, the study of which could potentially lead to a proof of the above conjecture. Thus it is beneficial to characterize finite type conditions for $A$ for the purposes of future study.

Suppose that $\gamma_2''\ne 0$; this implies that at least one of \eqref{piLfold} and \eqref{piRfold} must hold. Let 
\begin{align*}
    h^L_n(t)&=\det\Big(\mx{\gamma_2''(t) & \gamma_3''(t) \\ \gamma_2^{(n+2)}(t) & \gamma_3^{(n+2)}(t)}\Big) + \tfrac{n}2\gamma_2''(t)\gamma_2^{(n+1)}(t), \qquad n=1,2,... \\
    h^R_n(t)&=\det\Big(\mx{\gamma_2''(t) & \gamma_3''(t) \\ \gamma_2^{(n+2)}(t) & \gamma_3^{(n+2)}(t)}\Big) - \tfrac{n}2\gamma_2''(t)\gamma_2^{(n+1)}(t), \qquad n=1,2,... 
\end{align*}
Then $h^L_1\ne 0$ and $h^R_1\ne 0$ are equivalent to \eqref{piLfold} and \eqref{piRfold} respectively. Further, at a point $t_0$ where $h^L_1(t_0)\ne 0$ (and hence $\pi_L$ is a fold), the smallest $n$ for which $h^R_n(t_0)\ne 0$ is the type of $\pi_R$. Analogously, at a point $t_0$ where $h^R_1(t_0)\ne 0$, the smallest $n$ such that $h^L_n(t_0)\ne 0$ is the type of $\pi_L$. Although the $L^2$ regularity estimates of \cite{Co99} still apply in these cases, if $n>1$ there are currently no known non-trivial $L^p$ regularity estimates for $p\ne 2$.

\subsection{A Model Case}\label{modelcase}

Following Secco in \cite{sec99}, we consider as a model case the moment curve $\gamma(t)=(t,t^2,\alpha t^3)$. A brief computation show that $\gamma$ satisfies \eqref{piLfold} if and only if $\alpha\ne -1/6$ (in which case it satisfies \eqref{piLfold} for all $t$), and similarly satisfies \eqref{piRfold} if and only if $\alpha\ne 1/6$. Hence the conditions of Theorem \ref{mainthm} are satisfied if and only if $\alpha=1/6$. In this case we can make the following changes of variables. Let $\eta(y)=(y_2+y_1^2,y_3-\tfrac23 y_1^3-\tfrac12y_1y_2,y_1)$. This is a smooth function whose Jacobian always has determinant 1. We apply the operator $A$ to $f\circ \eta$ to obtain
\[
A(f\circ\eta)(x)=\int f(x_2+x_1^2-2x_1t,x_3-\tfrac23 x_1^3-\tfrac12 x_1x_2+2x_1^2t-x_1t,x_1-t) d\mu(t)
\]
Next, we change variables $(\tilde{x}_1,\tilde{x}_2,\tilde{x}_3)= (x_1,x_2-x_1^2,x_3-\tfrac12 x_1x_2+\tfrac13x_1^3)$ to get
\[
A(f\circ\eta)(\tilde{x})=\int f(\tilde{x}_2+2\tilde{x}_1(\tilde{x}_1-t),\tilde{x}_3-\tilde{x}_1(\tilde{x}_1-t)^2,\tilde{x}_1-t) d\mu(t).
\]
Finally, letting $y_3:=\tilde{x}_1-t$ we see that our operator has been transformed into the adjoint of a restricted $X$-ray transform of the type analyzed by Pramanik and Seeger in \cite{PrSe06}, associated to the curve $y_3\mapsto (-2y_3,y_3^2)$. Thus by applying Theorem 1.2 of \cite{PrSe06} to the adjoint of $A$ and changing variables back, we conclude that $A$ maps $L^p_{comp}(\rr^3)$ boundedly into $L^p_{1/p}(\rr^3)$ for $p>4$, exactly the statement of Theorem \ref{mainthm}. The observation that $A$ can be transformed into the adjoint of a restricted $X$-ray transform for this choice of $\gamma$ suggests that the sharpness examples found in \cite{PrSe06} could be used to prove the sharpness of Corollary \ref{corollary} for more general curves $\gamma$. Adapting these arguments allows us to show that for any curve $\gamma$ satisfying the conditions of Theorem \ref{mainthm} this region of boundedness is (possibly up to the horizontal dashed line) the best possible.

\begin{prop} \label{sharpnessprop}
    If $\gamma$ satisfies the conditions of Theorem \ref{mainthm} and $A:L^p_{comp}\to L^p_s$ then $s\le \min\{\tfrac{1}{p},\tfrac{1}{2}(1-\tfrac{1}{p})\}.$
\end{prop}

Proposition \ref{sharpnessprop} establishes the sharpness of the solid lines in the trapezoidal region described in Corollary \ref{corollary}, and thus proves the sharpness of Theorems \ref{mainthm} and \ref{heisensobocor} in the Sobolev exponent. The horizontal dashed line is sharp for examples in the class of averaging operators considered; in particular, using Proposition 1.1 from \cite{PrSe06} and the changes of variables above, we see that Corollary \ref{corollary} cannot be improved past the dashed line for the moment curve, and hence Theorems \ref{mainthm} and \ref{heisensobocor} cannot be improved in $p$ beyond the endpoint $p=4$. It is conjectured, but not yet proven that the horizontal dashed line in Corollary \ref{corollary} is sharp for all curves $\gamma$ satisfying the conditions of Theorem \ref{mainthm}. 

The layout of this paper is as follows. In Section \ref{conormalbundle} we analyze the behavior of the conormal bundle associated to $A$. In Section \ref{sharpness} we prove Proposition \ref{sharpnessprop} using the machinery from Section \ref{conormalbundle}. In Section \ref{InitialDecompositionSection} we begin the proof of Theorem \ref{mainthm} by relating it to an estimate on oscillatory integrals. This is the main estimate in the paper, and is proven in Sections \ref{L2section} and \ref{DecouplingSection} using respectively the Cotlar-Stein lemma and decoupling for the cone. In Section \ref{PRSSection} we finish the proof of Theorem \ref{mainthm} with a Calder\`on-Zygmund type estimate proven in \cite{PrRoSe11}, and also prove Theorem \ref{heisensobocor} and Corollary \ref{corollary}.

\subsection{Notation}

We denote $(x_1,x_2,x_3)=(x_1,x')$ and $(y_1,y_2,y_3)=(y_1,y')$. In this paper $C$ will denote a large constant, $C>1$, and $c$ will denote a small positive constant $0<c<1$. The values of both of these constants may change from line to line. Additionally, for non-negative quantities $X$ and $Y$ we will write $X\lesssim Y$ to denote the existence of a positive constant $C$ such that $X\le CY$. If this constant depends on a parameter such as $\varepsilon$ we write $X\lesssim_\varepsilon Y$. If $X\lesssim Y$ and $Y\lesssim X$ then we write $X\simeq Y$.

For ease of reading the dot $\cdot$ will be reserved for the inner product on $\rr^2$ and $\langle \ , \ \rangle$ for inner product on $\rr^3$. In cases where this choice affects readability we default to $\langle \ , \ \rangle$, but these instances should be clear from context. In this paper, $e_i$ for $i=1,...,n$ will represent the standard unit basis vectors in $\rr^n$.

Generally we will denote smooth bumps with variations of $\chi$, whereas cutoff functions supported on a set $E$ are denoted by $\1_E$.

\section*{Acknowledgements}

I would like to thank my advisor Andreas Seeger for introducing me to this problem and for many hours of fruitful discussion. I would also like to thank the anonymous referee, who provided several helpful suggestions.

\section{The Conormal Bundle}\label{conormalbundle}

Assume $f$ is compactly supported in $B_1(0)$ and let $\chi\in C_c^\infty(\rr)$ be equal to 1 on $[-\tfrac12,\tfrac12]$ and supported on $[-1,1]$. For the rest of this section, let the superscript $(\cdot)^\Phi$ denote the parametrization of a geometric object in the coordinate system induced by $\Phi$. We write $A$ as a generalized Radon transform associated to the manifold $M=M^\Phi$ defined in \eqref{ManifoldM}, i.e.
\begin{equation}\label{Agamma}
    A f(x)=\chi(x_1)\iiint e^{i\tau\cdot\Phi(x,y)} \chi(y_1) f(y)  \, d\tau_2 d\tau_3 dy.
\end{equation}  
The twisted conormal bundle of $M$ is given by 
\[
\mathcal{C}:=\{(x,\xi,y,-\eta) \, : \, (x,y)\in M, \, (\xi,\eta)\in \mathcal{N}^*_{(x,y)}M\}
\]
 In the coordinates induced by the defining function $\Phi$, $\mathcal{C}$ is given by
\[
    \mathcal{C}^\Phi:=\left\{\big(x,( \tau\cdot\Phi)_x,y,-(\tau\cdot\Phi)_y\big) \, : \, \Phi(x,y)=0\right\}\subset T^*\rr^3_x\times T^*\rr^3_y.
\]
Let $\pi_L:\mathcal{C}\to T^*\rr^3_x$ and $\pi_R:\mathcal{C}\to T^*\rr^3_y$ be projection maps. We define $\mathcal{L}\subset \mathcal{C}$ to be the conic submanifold 
\begin{equation}\label{Ldefn}
       \mathcal{L}:=\{P\in \mathcal{C} \ : \ \det(d\pi_L)_P=0\}.
\end{equation}

Since $\Phi$ parametrizes $M$ as a graph, i.e. $\Phi(x,y)=F(x,y_1)-y'$ for some smooth $F$ (see \eqref{Phi}), we can parametrize the manifold $\mathcal{C}^\Phi$ by $(x,y_1,\tau)$; by an abuse of notation let $P(x,y_1,\tau)\in\mathcal{C}^\Phi$ refer to the point $P\in \mathcal{C}^\Phi$ specified by the parameters $(x,y_1,\tau)$. Since $(\tau\cdot\Phi)_x$ and $(\tau\cdot\Phi)_y$ are functions of $(x,y_1,\tau)$, the differentials of the projections $\pi_L$ and $\pi_R$ can be expressed as the Jacobians of the functions $\pi_L^\Phi:(x,y_1,\tau)\mapsto (x,(\tau\cdot \Phi)_x)$ and $\pi_R^\Phi:(x,y_1,\tau)\mapsto (y_1,F(x,y_1),(\tau\cdot\Phi)_y)$, respectively. A calculation yields that 
\[
    \det(d\pi_L^\Phi)=-\det(d\pi_R^\Phi)=(\tau_2-\tfrac12\tau_3x_1)\gamma_2''(x_1-y_1)+\tau_3\gamma_3''(x_1-y_1).
\]
Then in the coordinates induced by $\Phi$,  
\begin{align}\label{S1defn1}
    \mathcal{L}^\Phi:=\{P(x,y_1,\tau)\in\mathcal{C}^\Phi \, &: \, \tau_2=\rho(\gamma_3''(x_1-y_1)-\tfrac{1}{2} x_1\gamma_2''(x_1-y_1)) \\
    & \qquad \qquad \qquad \tau_3=-\rho\gamma_2''(x_1-y_1), \, \rho\in\rr\}.\notag
\end{align}

Next, we recall the definition of a Whitney fold and blowdown, as described in the survey paper \cite{GrSe02}.

\begin{defn}\label{foldblowdown}
    Suppose $g:X\to Y$ is a $C^\infty$ map between $C^\infty$ manifolds of corank $\le 1$, and the set $\mathcal{L}=\{P\in X \ : \ \det (dg)_P=0\}$ is an immersed hypersurface. 
    We say $V$, a nonzero smooth vector field on $X$, is a {\bf\emph{kernel field}} of $g$ if $V|_P\in \Ker(dg)_P$ for all $P\in \mathcal{L}$.   
    
    We say $g$ is a {\bf\emph{Whitney fold}} if 
    for every kernel field $V$ of $g$ and every $P\in\mathcal{L}$ we have $V(\det dg)\ne 0$ at $P$. 
    
    We say $g$ is a {\bf\emph{blowdown}} if every kernel field $V$ of $g$ is everywhere tangential to $\mathcal{L}$. Note this implies that $V^k(\det dg)\big|_P=0$ for all $k\in\nn$ and all $P\in\mathcal{L}$. 
\end{defn}

It useful to note that two kernel fields for a map $g$ only differ by a smooth function \cite{GrSe02}, so it suffices to check these conditions for one explicit kernel field for a given map. Moreover, since the definitions of Whitney folds and blowdowns are geometric in nature, they are invariant under diffeomorphisms, which we will use to our advantage. A kernel field for $\pi_L^\Phi$ is given (in the coordinates induced by $\Phi$) by $V_L^\Phi=\partial_{y_1}-\tfrac12\tau_3\partial_{\tau_2},$ and we see that by \eqref{S1defn1} we have for any point $P(x,y_1,\tau)\in\mathcal{L}^\Phi$  
\[
    V_L^\Phi\det d\pi_L^\Phi\Big|_P=\rho\Big[-\det\left(\mx{\gamma_2''(x_1-y_1) & \gamma_3''(x_1-y_1) \\ \gamma_2'''(x_1-y_1) & \gamma_3'''(x_1-y_1)}\right)-\tfrac12(\gamma_2''(x_1-y_1))^2\Big].
\]
Similarly, 
a kernel field for $\pi_R^\Phi$ is given by 
\[
    V_R^\Phi=\bigg\langle\bigg(\mx{1 \\ \gamma_2'(x_1-y_1) \\ \gamma_3'(x_1-y_1)-\tfrac12\gamma_2(x_1-y_1)-\tfrac12(x_1+y_1)\gamma_2'(x_1-y_1)}\bigg), \nabla_x\bigg\rangle,
\]
and as above we have for $P(x,y_1,\tau)\in\mathcal{L}^\Phi$
\[
    V_R^\Phi\det d\pi_R^\Phi\Big|_P=\rho\left[-\det\left(\mx{\gamma_2''(x_1-y_1) & \gamma_3''(x_1-y_1) \\ \gamma_2'''(x_1-y_1) & \gamma_3'''(x_1-y_1)}\right)+\tfrac12\left(\gamma_2''(x_1-y_1)\right)^2\right].
\]
Thus \eqref{piLfold} and \eqref{piRfold} are precisely the conditions under which $\pi_L$ and $\pi_R$ are folds, respectively.

The assumptions of Theorem \ref{mainthm} restrict the class of admissible curves $\gamma$ quite significantly, as it implies we can rewrite
\begin{equation*}
    \det\left(\mx{\gamma_2'' & \gamma_3'' \\ \gamma_2''' & \gamma_3'''}\right)=\tfrac12(\gamma_2'')^2
\end{equation*}
as $(\gamma_3''/\gamma_2'')'=\tfrac12$. This implies the existence of constants $C_1,C_2,C_3\in \rr$ such that 
\[
    \gamma_3(t)=(\tfrac12t+C_1)\gamma_2(t)-\Gamma(t)+C_2t+C_3,
\]
where 
\begin{equation}
    \Gamma(t)=\int_0^t \gamma_2(s) \, ds. \label{Gammadefn}
\end{equation}
Assuming that \eqref{piRfold} vanishes uniformly also implies that $\pi_R$ is a blowdown. Indeed, the calculation above, along with \eqref{S1defn1}, show that $\mathcal{L}^\Phi$ is given by the set
\begin{align}
    \{P(x,y_1,\tau)\in \mathcal{C}^\Phi \, : \, (\tau_2,\tau_3)=
    \rho(\tfrac12 y_1-C_1, 1), \ \rho\in\rr\}. \label{S1defn}
\end{align} 
Clearly any vector field spanned by $\{\partial_{x_j}\}_j$ (including $V_R$) is tangent to $\mathcal{L}$. 

 For the rest of this paper we will work with a slightly modified version of $A$. First, as the Schwartz kernel of $A$ is 0 away from the set $\{\Phi=0\}$, we can substitute $x_2-y_2=\gamma_2(x_1-y_1)$ into $\Phi^3(x,y)$ to get
 \[
 \Phi^3(x,y)\Big|_{\{\Phi(x,y)=0\}}=x_3-y_3-\gamma_3(x_1-y_1)+\tfrac12(x_1+y_1)\gamma_2(x_1-y_1)-\tfrac12x_2x_1-\tfrac12 y_2y_1.
 \] 
 By an abuse of notation let us call the above $\Phi^3(x,y)$. Next, by making a smooth change of variables 
\begin{align*}
    g_1:(x_1,x_2,x_3)&\mapsto(x_1+C_1,x_2,x_3+C_2x_1+\tfrac12(x_1+C_1)x_2)+C_3 \\
    g_2:(y_1,y_2,y_3)&\mapsto (y_1+C_1,y_2,y_3+C_2y_1+\tfrac12(y_1+C_1)y_2)
\end{align*}
we can reduce the proof of Theorem 1 to analyzing regularity estimates of the operator
\begin{equation}\label{newoperatordefn}
    \A f(x)=\chi(x_1)\int e^{i\tau\cdot\tilde{\Phi}(x,y)} \chi(y_1)f(y) \, dy d\tau, 
\end{equation}
where 
\begin{equation}
    \tilde{\Phi}(x,y):=\Phi(g_1^{-1}(x),g_2^{-1}(y))=x'-y'+S(x_1,y_1)
\end{equation}
and
\begin{align}
    S^2(x_1,y_1)&=-\gamma_2(x_1-y_1) \label{Sdefn} \\
    S^3(x_1,y_1)&=y_1\gamma_2(x_1-y_1)+\Gamma(x_1-y_1), \notag
\end{align}
where $\Gamma(t)$ is given in \eqref{Gammadefn}. In the coordinates induced by $\tilde{\Phi}$ the fold and blowdown behavior manifests in the behavior of various derivatives of $\tau\cdot S(x_1,y_1)$. 
Indeed, calculating $\pi_L^{\tilde{\Phi}}$ and $\pi_R^{\tilde{\Phi}}$ at a point $P(x,y_1,\tau)\in\mathcal{C}^{\tilde{\Phi}}$ yields
\[
    \det(d\pi_L^{\tilde{\Phi}})=-\det(d\pi_R^{\tilde{\Phi}})=\tau\cdot S_{x_1y_1}(x_1,y_1)=\gamma_2''(x_1-y_1)(\tau_2-y_1\tau_3),
\]
meaning that 
\begin{equation}\label{Lparametrization}
    \mathcal{L}^{\tilde{\Phi}}=\{P(x,y_1,\tau)\in\mathcal{C}^{\tilde{\Phi}} \, : \,  \tau_2-y_1\tau_3=0 \}.
\end{equation}
Moreover, $V_L^{\tilde{\Phi}}=\partial_{y_1}$ and $V_R^{\tilde{\Phi}}=\partial_{x_1}-S_{x_1}(x_1,y_1)\cdot \nabla_{x'}$, so we have for $k_1,k_2\ge 1$ and any point $P(x,y_1,\tau)\in\mathcal{L}^{\tilde{\Phi}}$,
\begin{align*}
    \big(V_L^{\tilde{\Phi}}\big)^{k_1}\big(V_R^{\tilde{\Phi}}\big)^{k_2}\det(d\pi_{L}^{\tilde{\Phi}})\big|_P=\partial_{x_1}^{k_2+1}\partial_{y_1}^{k_1+1}\big[\tau\cdot S(x_1,y_1)\big].
\end{align*}
A similar expression holds for $d\pi_R^{\tilde{\Phi}}$. For these reasons, we will present various derivatives of $\tau\cdot S(x_1,y_1)$ here for future reference. For $j\ge 1$,
\begin{align}
    \label{Sxjy}\partial_{x_1}^j\partial_{y_1} \big[\tau\cdot S(x_1,y_1)\big]&=\gamma_2^{(j+1)}(x_1-y_1)(\tau_2-y_1\tau_3) \\
    \label{Sxjy2}\partial_{x_1}^j\partial_{y_1}^2\big[\tau\cdot S(x_1,y_1)\big]&=-\gamma_2^{(j+2)}(x_1-y_1)(\tau_2-y_1\tau_3)-\tau_3\gamma_2^{(j+1)}(x_1-y_1).
\end{align}
In particular we observe from these formulas that for $P(x,y_1,\tau)\in\mathcal{L}^{\tilde{\Phi}}$
\begin{align}
    \label{newL} \det(d\pi_L^{\tilde{\Phi}})\Big|_P&=\tau\cdot S_{x_1 y_1}(x_1,y_1)\Big|_{(\tau_2,\tau_3)=\rho(y_1,1)}=0 \\
    \label{newfold} V_L^{\tilde{\Phi}}\det(d\pi_L^{\tilde{\Phi}})\Big|_{P}&=\tau\cdot S_{x_1 y_1^2}(x_1,y_1)\Big|_{(\tau_2,\tau_3)=\rho(y_1,1)}=-\rho\gamma_2''(x_1-y_1) \\
    \label{newblowdown}(V_R^{\tilde{\Phi}})^j\det(d\pi_R^{\tilde{\Phi}})\Big|_{P}&=-\tau\cdot S_{x_1^{j+1}y_1}(x_1,y_1)\Big|_{(\tau_2,\tau_3)=\rho(y_1,1)}=0,  \qquad \forall j\in\nn,
\end{align}
which respectively encode the definition of $\mathcal{L}$, the fold condition on $\pi_L$, and an implication of the blowdown condition on $\pi_R$.

As is shown in Section 3 of \cite{PrSe19}, the condition that $\pi_L$ is a fold is enough to ensure a curvature condition on the fibers of $\mathcal{L}$, as formulated by Greenleaf and Seeger in \cite{GrSe94}. Let
\begin{equation}\label{Sigmaxdefn}
    \Sigma_x:=\{\xi\in\rr^3 \ : \ (x,\xi)\in \pi_L(\mathcal{L})\}.
\end{equation}
We see by the definition of $\mathcal{L}^{\tilde{\Phi}}$, in the coordinates induced by $\tilde{\Phi}$ the set \eqref{Sigmaxdefn} is given by
\begin{equation}\label{Sigmax}
    \Sigma_x^{\tilde{\Phi}}=\left\{\rho\left(\mx{\gamma_2(x_1-y_1) \\ y_1 \\ 1}\right) \, : \, \rho\in\rr, \ y_1\in \supp\,\chi \right\}
\end{equation}
Since $\gamma_2''(t)\ne 0$ we see $\Sigma_x^{\tilde{\Phi}}$ clearly has one non-vanishing principal curvature, hence Bourgain-Demeter-Wolff decoupling for the cone can be applied (see \cite{BoDe15}, \cite{Wo00}, and also \cite{PrSe07}). It is important to note here that the fibers $\Sigma_x^{\tilde{\Phi}}$ vary with $x_1$. This behavior contrasts with the situation in \cite{PrSe06} where the fibers of $d\pi_R$ were constant in $x$ and thus decoupling could be applied directly to the fixed cone $\Sigma$ for all $x$. Instead we use ideas from the decoupling estimate in Section 4 \cite{AnClPrSe2018}, localizing $x_1$ then applying decoupling iteratively at smaller and smaller scales, applying changes of variables at each step so the cone $\Sigma_{\tilde{\Phi}}$ varies less and less with $x$. We will explain this approach in Section \ref{DecouplingSection}.

\section{Sharpness}\label{sharpness}

We prove Proposition \ref{sharpnessprop}, which establishes the sharpness of Corollary \ref{corollary} and Theorems \ref{mainthm} and \ref{heisensobocor}. Consider a Fourier multiplier $m_k$ in $\rr^2$ of order $0$ which vanishes for $|\xi'|\le c 2^{k}$ (here $\xi'=(\xi_2,\xi_3)$ so we identify the multiplier $m_k$ as acting on functions in $\rr^3$ in the second and third coordinates). Observe that $\A$ commutes with $m_k(D')$. So if $\A: L^p\to L^p_s$ for some $p\in(1,\infty)$ it follows that
\[
    \|m_k(D') \A f\|_p\le C_p 2^{-ks} \|f\|_p.
\]
Observe as well that the same result holds for the adjoint of $\A$, given by
\[
\A^*g(y)=\chi(y_1)\int e^{-i\tau\cdot\tilde{\Phi}(x,y)} \chi(x_1) g(x) \, dx;
\]
if $\A:L^p\to L^p_s$ then $\A^*:L^{p'}_{-s}\to L^{p'}$, and $\|m_k(D') \A f\|_{p'}\le C_p 2^{-ks} \|f\|_{p'}$. As discussed in Section \ref{modelcase}, we borrow heavily from the sharpness results in \cite{PrSe06}.

\subsection{$s\le 1/p$}

Let $\zeta_1$ be supported in $\{\xi' \, : \, 1/2\le |\xi'|\le 2\}$ with $\widehat{\zeta}_1(0)=1$. Let $m_k$ be the Fourier multiplier given by $\zeta_1(2^{-k}\xi')$, acting on functions in $\rr^3$. Then
\begin{equation}
    m_k(D')\A^* f(y_1,y')=\chi(y_1)\int 2^{2k} \widehat{\zeta}_1(2^k(y'-x'-S(x_1,y_1))) f(x) \chi(x_1) \, dx. \label{sharpness1integral}
\end{equation}
Let $x_0\in [-1/2,1/2]$ such that $\chi(x_0)> 0$, and choose $k$ large enough that $\chi(x_1)>c>0$ for $|x_1-x_0|\le 2^{-k}$. Let $f_k$ be the indicator function of a ball of radius $2^{-k}$ centered at $(x_0,0,0)$, and let $c_{y_1}$ be the curve $\{S(x_1,y_1) \, : \, x_1\in \supp\, \chi\}\subset\rr^2$. For small $\varepsilon>0$ let $E_{y_1}$ be the set of all $y'$ such that $\dist(y',c_{y_1})\le\varepsilon 2^{-k}$. Since $\gamma_2''\ne 0$ on $[-,1,1]$ we can conclude that $S(\cdot,y_1)$ is a regular curve in $\rr^2$ on a neighborhood of $x_0$ that has diameter at least $1/2$, hence we estimate $|E_{y_1}|\approx 2^{-2k}$ for each fixed $y_1$. As $\widehat{\zeta}_1$ is positive near the origin we see that the integrand in \eqref{sharpness1integral} is bounded below by $c2^{2k}$ if $y'\in E_{y_1}$, whence we can bound the integral \eqref{sharpness1integral} below by $2^{-k}$. After integrating in $y'$ over the size of $E_{y_1}$ and in $y_1$ over a fixed compact set, we see that $\|m_k(D')\A^* f_k\|_{p'}\gtrsim 2^{-k}2^{-2k/{p'}}$. On the other hand, $\|f_k\|_{p'}\lesssim 2^{-3k/p}$, hence we must have $s\le 1-1/p'=1/p$.

\subsection{$s\le \tfrac12(1-\tfrac{1}{p})$}

Notice that the direction of the vector $S_{x_1 y_1}(x_1,y_1)=\gamma_2''(x_1-y_1)(1,-y_1)$ does not depend on $x_1$. Let $T(y_1)=(1,-y_1)$ and let $N(y_1)=(y_1,1)$. Let $\zeta_2\in\mathcal{S}(\rr)$ be such that $\widehat{\zeta}_2$ is non-negative everywhere and is positive in $[-1/2,1/2]$. Let $\zeta_3$ be supported in $\{1/2\le |t|\le 2\}$ with $\widehat{\zeta}_3\ge 1/2$ on $[-C,C]$. Pick $b$ such that $\chi(b)>0$ and define the Fourier multiplier $m_k$ by
\[
    m_k(\tau_2,\tau_3)=\zeta_2(2^{-k/2}\langle \tau,T(b)\rangle) \zeta_3(2^{-k}\langle \tau,N(b)\rangle).
\]
Again, $m_k$ acts on functions in $\rr^3$ as
\[
m_k(D')f(x)=\mathfrak{F}^{-1}\Big[m_k(\tau_2,\tau_3)\widehat{f}(\xi_1,\tau_2,\tau_3)\Big].
\] 
Since $m_k(\tau)$ vanishes for $|\tau|\le c 2^k$ we have $\|m_k(D')\A^* f\|_{p'}\le 2^{-ks} \|f\|_{p'}$, and that
\begin{align}
    m_k(D')\A^* f(y)=& \chi(y_1) \int 2^{3k/2} \widehat{\zeta}_2(2^{k/2}\langle y'-x'-S(x_1,y_1),T(b)\rangle ) \label{sharpness2integral}\\
    & \qquad \qquad \times \widehat{\zeta}_3(2^k \langle y'-x'-S(x_1,y_1),N(b)\rangle) f(x) \, dx. \notag
\end{align}
Let $g_k$ be the indicator of the set defined by the equations
$|\langle x'+S(x_1,b),T(b)\rangle|\le 2^{-k/2}$, $|\langle x'+S(x_1,b),N(b)\rangle|\le 2^{-k}$, and $x_1\in I$.
Let $P_k$ be the set of $y$ such that $|\langle y',T(b)\rangle|\le 2^{-k/2}$, $|\langle y',N(b)\rangle|\le 2^{-k}$, and $|y_1-b|\le 2^{-k/2}$.
For $x\in \supp\, g_k$ and $y\in P_k$ we see that since $|y_1-b|\le 2^{-k/2}$,
\[
    |\langle y'-x'-S(x_1,y_1),T(b)\rangle |\le C 2^{-k/2}.
\]
However, we have better decay in the $N(b)$ direction, as $S(x_1,\cdot)$ vanishes to second order in the $N(b)$ direction. Indeed, a Taylor expansion reveals
\[
    |\langle S(x_1,y_1)-S(x_1,b),N(b)\rangle|=|(y_1-b)^2(-\gamma_2'(x_1-b))+|y_1-b|^2R_1(x_1,y_1)|\le C 2^{-k},
\]
where $R_1(x_1,y_1)$ is smooth and uniformly bounded. Thus $|\langle y'-x'-S(x_1,y_1),N(b)\rangle|\le C2^{-k}$, implying by the conditions on $ \widehat{\zeta}_2$ and $\widehat{\zeta}_3$ that the integrand in \eqref{sharpness2integral} is greater than $c2^{3k/2}$, implying that $m_k(D')\A^* g_k(y)$ is bounded below by a positive constant for all $y\in P_k$. Thus 
$\|m_k(D')\A^*g_k\|_{p'}\ge c 2^{-2k/p'}.$
On the other hand, $\|g_k\|_{p'}\le 2^{-3k/2p'}$, implying that $s\le \tfrac{1}{2p'}=\tfrac12(1-\tfrac{1}{p})$.

\section{Initial Decomposition}\label{InitialDecompositionSection}

We localize in $|\tau|$ then localize away from the singular variety $\mathcal{L}$, following the ideas of Phong and Stein in \cite{PhSt91}. Let $\chi_1\in C_c^\infty(\rr)$ be equal to 1 on $[\tfrac{1}{2},2]$ and supported on $[\tfrac{1}{4},4]$ such that $\sum_{k\in\zz} \chi_1(2^k\cdot)\equiv 1$. For $k\ge 1$ define $\chi_k(|\tau|)=\chi_1(2^{1-k}|\tau|)$, and for $k=0$ define $\chi_0(|\tau|)=\sum_{k\le 0} \chi_1(2^{1-k}|\tau|)$. For $0\le \ell\le k/2$ let
\[
    a_{k,\ell,\pm}(y_1,\tau)=\left\{
    \begin{array}{cc}
         &  \chi_1(2^{\ell-k}(\pm (\tau_2-y_1\tau_3))) \qquad \ell< k/2 \\
         & 1-\sum_{k> 2\ell} \chi_1(2^{\ell-k}(\tau_2-y_1\tau_3)) \qquad \ell=\lfloor k/2 \rfloor
    \end{array}\right.
\]
and define
\begin{align}
     \A_{k,\ell,\pm}f(x)&=\chi(x_1)\int e^{i\tau\cdot\tilde{\Phi}(x,y)} \chi(y_1)f(y) \chi_k(|\tau|)a_{k,\ell,\pm}(y_1,\tau) \, dy\,d\tau. \label{Akl}
\end{align}
We will suppress the dependence on $\pm$. We prove the following estimate.
\begin{prop}\label{decomposedfinal}
    For $p>4$ there exists $\varepsilon_0(p)>0$ such that for all $\ell\le \lfloor k/2\rfloor$,
    \[
        \|\A_{k,\ell}\|_{L^p\to L^p}\le C_p 2^{-(k+\ell\varepsilon_0)/p}.
    \]
\end{prop}
This proposition follows by interpolation with $L^2$ estimates, $L^\infty$ estimates, and a decoupling inequality. Let $\{\nu\}$ be a set of $2^{-\ell}$-separated points in the unit interval, and let $I_\nu=[\nu,\nu+2^{-\ell}]$. Then for a function $f(y_1,y')$ supported in the unit cube, let $f_{\nu}(y):=f(y)\1_{I_\nu}(y_1)$, so that $f=\sum_{\nu} f_\nu$ with almost disjoint supports in $y_1$. The necessary $L^2$ estimate is the following.

\begin{prop}\label{L2boundA} 
    Let $\A_{k,\ell}$ be defined as above. 
    \begin{align}
        \|\A_{k,\ell}\|_{L^2\to L^2} &\lesssim  2^{(\ell-k)/2}, \qquad  \ell\le k/2. \label{L2est}
    \end{align}
    Moreover,
    \begin{align}
        \Big(\sum_\nu \|\A_{k,\ell} f_\nu\|_{L^2}^2\Big)^{1/2} &\lesssim  2^{(\ell-k)/2}\Big(\sum_\nu\|f_\nu\|_{L^2}^2\Big)^{1/2}, \qquad  \ell\le  k/2, \label{l2L2est}.
    \end{align}
\end{prop}

Proposition \ref{L2boundA} will be proven in Section \ref{L2section} following methods of almost-orthogonality found in the proof of the Calder\'on-Vaillancourt theorem (see \cite{MuSc13}, \S \, 9.2), originally introduced into this context by Phong and Stein \cite{PhSt91} and Cuccagna \cite{Cu97}. After that we prove the following decoupling inequality.

\begin{prop}\label{decoupling}
    For every $\varepsilon>0$ there exists $N>0$ such that
    \begin{align*}
        \Big\|\sum_\nu \A_{k,\ell} f_\nu\Big\|_{L^p}&\lesssim_\varepsilon 2^{\ell(1/2-1/p+\varepsilon)}\Big(\sum_\nu \|\A_{k,\ell} f_\nu\|_{L^p}^p\Big)^{1/p}+2^{-kN}\|f\|_{L^p}
    \end{align*}
    for $2\le p\le 6$ and $\ell\le k/2$.
\end{prop}

Following a similar approach to \cite{AnClPrSe2018} and \cite{PrSe19}, we prove Proposition \ref{decoupling} in Section \ref{DecouplingSection} using induction. At each step we combine $l^p$ decoupling with suitable changes of variables. Note that in \cite{AnClPrSe2018} and \cite{PrSe19} similar decoupling estimates was proven, but only for $\ell\le k/3$; we prove a larger range of $\ell$ because our operator has only one fold singularity.

\begin{proof}[Proof that Propositions \ref{L2boundA} and \ref{decoupling} imply Proposition \ref{decomposedfinal}]

    We begin by proving an $L^\infty$ estimate for $\A_{k,\ell}$, namely that
    \begin{align}
        \sup_{\nu} \|\A_{k,\ell}f_\nu\|_\infty&\lesssim 2^{-\ell} \sup_\nu \|f_\nu\|_\infty \label{linftyLinfty}\\
        \|\A_{k,\ell} f\|_\infty&\lesssim \|f\|_\infty. \label{Linfty}
    \end{align}

    First, by \eqref{Akl} we see that for fixed $y_1$ and any $N_1,N_2\in\nn$,
    \begin{align*}
        (\partial_{\tau_2}-y_1\partial_{\tau_3})^{N_1}[\chi_k(|\tau|)a_{k,\ell,\pm}(y_1,\tau)]&\lesssim_{N_1} 2^{(\ell-k)N_1} \\
        (y_1\partial_{\tau_2}+\partial_{\tau_3})^{N_2}[\chi_k(|\tau|)a_{k,\ell,\pm}(y_1,\tau)]&\lesssim_{N_2} 2^{(-k)N_2}.
     \end{align*}
    Thus we integrate by parts with respect to these directions in $\tau$, garnering for any $N>0$
    \begin{align*}
        |\A_{k,\ell}f_\nu(x)|&\le \sup_x \|f_\nu\|_\infty \int_{|y_1-\nu|\le 2^{-\ell}} \iint_{\supp\,  \chi_k a_{k,\ell,\pm}} C_N \\
        & \qquad \times \frac{1}{(1+2^{k-\ell}|\tilde{\Phi}^2(x,y)-y_1\tilde{\Phi}^3(x,y)|)^N} \\
        & \qquad \times \frac{1}{(1+2^{k}|y_1\tilde{\Phi}^2(x,y)+\tilde{\Phi}^3(x,y)|)^N} \, d\tau\,dy'\,dy_1.
    \end{align*}
    Since the size of the support of $\chi_k(|\tau|)a_{k,\ell,\pm}(y_1,\tau)$ in $\tau$-space (for fixed $y_1$) is $2^{k-\ell}$ in the $(1,-y_1)$ direction and $2^k$ in the $(y_1,1)$ direction, integrating in $\tau$ and $y'$ gains a constant independent of $x,\ell,$ and $k$. Finally, integrating in $y_1$ yields the desired bounds. 

    Interpolating \eqref{linftyLinfty} with \eqref{l2L2est} we obtain
    \begin{align}
        \Big(\sum_\nu \|\A_{k,\ell} f_\nu\|_p^p\Big)^{1/p} &\lesssim 2^{\ell(3/p-1)}2^{-k/p}\Big(\sum_\nu \|f_\nu\|_p^p\Big)^{1/p}, \qquad 2\le p\le \infty.
    \end{align}
    Combining this estimate with Proposition \ref{decoupling} we obtain
    \begin{equation}\label{epsilonnegative}
        \|\A_{k,\ell} f\|_p\lesssim_\varepsilon 2^{\ell(\varepsilon+2/p-1/2)}2^{-k/p}\Big(\sum_{\nu} \|f_\nu\|_p^p\Big)^{1/p} + 2^{-kN}\|f\|_p, \qquad 2\le p\le 6.
    \end{equation}
    Note that the power of $2^\ell$ in \eqref{epsilonnegative} is negative if $4<p\le 6$ and $\varepsilon$ is sufficiently small. A further interpolation with the $L^\infty$ estimate \eqref{Linfty} yields Proposition \ref{decomposedfinal} for $p>4$. 

\end{proof}

\section{$L^2$ Estimates}\label{L2section}

Define the oscillatory integral operator
\begin{equation}\label{Hdefn}
    H^{k,\ell}g(x_1; \tau):=\int e^{i\tau\cdot S(x_1,y_1)} a_{k,\ell}(y_1,\tau)\chi(y_1) \chi_k(|\tau|)g(y_1,\tau) \, dy_1.
\end{equation}
To prove Proposition \ref{L2boundA} it suffices to prove a uniform bound in $\tau$ for $H^{k,\ell}$ on $L^2(\rr)$.
\begin{prop}\label{L2boundH}
    There exists a constant $C>0$ such that for all $\tau\in\rr$ 
    \begin{equation}
        \int|H^{k,\ell}g(x_1;\tau)|^2 \, dx_1\le C2^{\ell-k} \int|g(y_1,\tau)|^2 \, dy_1.    
    \end{equation}
\end{prop}
To see why this suffices, let $\mathfrak{F}_{2,3}$ denote the Fourier transform in the second and third variables. Then because the phase for the kernel of $\A_{k,\ell}$ has the form
\[
\tau\cdot\tilde{\Phi}(x,y)=\tau_2(x_2-y_2)+\tau_3(x_3-y_3)+\tau\cdot S(x_1,y_1)
\]
we see that $\mathfrak{F}_{2,3}^{-1}\big[H^{k,\ell}(\mathfrak{F}_{2,3}f)\big](x_1,x')=\A_{k,\ell}f(x)$. 
Hence by two applications of Plancherel the estimate in Proposition \ref{L2boundH} implies the estimate in Proposition \ref{L2boundA}. For the rest of the section we fix $k,\ell$, and suppress the dependence of $H^{k,\ell}$ on $k$ and $\ell$ by writing $H:=H^{k,\ell}$.

Let $\C>0$ be a large constant to be picked later (independently of $k$ and $\ell$). At a loss of a finite constant, we can assume that $|x_1|,|y_1|\le \C^{-1}$. For $m,n\in\zz$ define disjoint rectangles $q_{m,n}$ to be the set of points $(x_1,y_1)$ satisfying
\begin{align}
    0 &\le \C 2^{k-2\ell}x_1-m \le 1 \\
    0 &\le 2^{\ell}y_1-n \le 1.
\end{align}
Thus $q_{m,n}$ is a rectangle indexed by $(m,n)\in\zz\times\zz$ aligned with the coordinate axes whose side parallel to the $x_1$-axis has length $\C^{-1}2^{2\ell-k}$ and whose side parallel to the $y_1$-axis is $2^{-\ell}$. Note that for $\ell\ge k/3$ the boxes $q_{m,n}$ are longer in the $x_1$ direction, reflecting the fact that under the blowdown condition on $\pi_R$ we don't expect much orthogonality in the $x_1$ direction. For each $m,n\in\zz$ let $\chimn_m(x_1)$ and $\chimn_n(y_1)$ be smooth cutoffs such that $\chimn_m(x_1)\chimn_n(y_1)$ is equal to 1 on $q_{m,n}$ and is supported on its double, and furthermore such that $\sum_{m,n}\chimn_m(x_1)\chimn_n(y_1)=1$.

Define $H_{m,n}:=\chimn_m(x_1)H[\chimn_n(y_1)\cdot]$ so that $\sum_{m,n} H_{m,n}=H$. By splitting our operator $H$ into a finite number of collections of $\{H_{m,n}\}$ we may assume that if $m\ne \tilde{m}$ then $|m-\tilde{m}|>2\C$ and if $n\ne \tilde{n}$ then $|n-\tilde{n}|> 2\C$.


We prove that the operators $H_{m,n}$ are almost orthogonal by the following estimates.
\begin{lem}\label{cotlarsteinestimates}
    If $|n-\tilde{n}|>2\C$ then
    \begin{equation}\label{disjoint1}
            \int |H_{m,n}H^*_{\tilde{m},\tilde{n}}g(x_1,\tau)|^2 \, dx_1=0.
    \end{equation}
    If $n=\tilde{n}$ then for every $N>0$
    \begin{align}\label{HHstar}
        \left(\int |H_{m,n}H^*_{\tilde{m},\tilde{n}}g(x_1,\tau)|^2 \, dx_1\right)^{1/2}\lesssim_N 2^{\ell-k}(1+|m-\tilde{m}|)^{-N}\Big(\int |g(y_1,\tau)|^2 \, dy_1\Big)^{1/2}.
    \end{align}
    If $|m-\tilde{m}|>2\C$ or if $m=\tilde{m}$ and $|n-\tilde{n}|>2\C$ then
    \begin{equation}\label{HstarH}
            \int |H_{m,n}^*H_{\tilde{m},\tilde{n}}g(y_1,\tau)|^2 \, dx_1=0.
    \end{equation}
\end{lem}

By an easy application of the Cotlar-Stein Lemma the estimates in Lemma \ref{cotlarsteinestimates} imply Proposition \ref{L2boundH}. We remark that since we only need the operator norms of $H_{m,n}H^*_{\tilde{m},\tilde{n}}$ and $H^*_{m,n}H_{\tilde{m},\tilde{n}}$ to be summable we can prove the above Lemma (and thus Proposition \ref{L2boundA}) under the weaker assumption that $\gamma(t)$ is $C^5$. 

\begin{proof}
First, we replace $\tau=2^k\tilde{\tau}$ so that $|\tilde{\tau}|\simeq 1$. 

The Schwartz kernel of $H_{m,n}H^*_{\tilde{m},\tilde{n}}$ is given by 
\begin{equation*}
    K_{m,n,\tilde{m},\tilde{n}}(x_1,w_1,\tilde{\tau})= \int e^{i2^k\tilde{\tau}\cdot (S(x_1,y_1)-S(w_1,y_1))} \sigma(x_1,w_1,y_1,\tilde{\tau})\, dy_1,
\end{equation*}
where the amplitude $\sigma$ is given by
\begin{align*}
    \sigma(x_1,w_1,y_1,\tilde{\tau})&= |\chi_k(|\tilde{\tau}|)\chi(y_1)|^2\chi(x_1)a_{k,\ell}(y_1,2^k\tilde{\tau})\chimn_m(x_1)\chimn_n(y_1) \\
      & \qquad \times\overline{\chi}(w_1)\overline{a_{k,\ell}}(y_1,2^k\tilde{\tau})\overline{\chimn_{\tilde{m}}}(w_1)\overline{\chimn_{\tilde{n}}}(y_1). 
\end{align*}
Thus if $|n-\tilde{n}|>2\C$, $\sigma=0$ and the kernel vanishes, proving \eqref{disjoint1}. 

Assume $n=\tilde{n}$. Since the sizes of the supports of $\chimn_m(x_1)$ and $\chimn_{\tilde{m}}(w_1)$ are both $\C^{-1} 2^{2\ell-k}$, the estimate \eqref{HHstar} follows from the estimate
\begin{equation}\label{HHstarkernel}
    |K_{m,n,\tilde{m},n}(x_1,w_1,\tilde{\tau})|\le C_N 2^{-\ell}(1+|m-\tilde{m}|)^{-N}\chimn_{m}(x_1)\chimn_{\tilde{m}}(w_1)
\end{equation}
by an application of Schur's test.

First, if $m=\tilde{m}$ then \eqref{HHstarkernel} holds because the size of support of $\chimn_n(y_1)$ is $2^{-\ell}$. 
For the case $|m-\tilde{m}|>2\C$ we use a nonstationary phase argument. Define the operator
\begin{align*}
        \M_{y_1}f&=\frac{1}{2^k\tilde{\tau}\cdot(S_{y_1}(x_1,y_1)-S_{y_1}(w_1,y_1))} \partial_{y_1}f.
\end{align*}
To apply an integration by parts argument with $\M_{y_1}$ we need to carefully analyze the $y_1$ derivative of the phase of $K_{m,\tilde{m}}$. A Taylor approximation yields
\begin{align}
    \tilde{\tau}\cdot\partial_{y_1}[S(x_1,y_1)-S(w_1,y_1)]=& \sum_{j=1}^K\frac{1}{j!}\tilde{\tau}\cdot S_{x_1^j y_1}(w_1,y_1)(x_1-w_1)^j+|x_1-w_1|^{K+1}r(x_1,y_1,\tilde{\tau}), \notag
\end{align}
where $r(x_1,y_1,\tilde{\tau})$ is a smooth function with bounded derivatives independent of $k$ and $\ell$. Via \eqref{Sxjy} we see that as long as $\C$ is sufficiently large (independently of $k$ and $\ell$) the first term dominates the rest, and
\begin{equation} \label{lowerboundy1}
    |\tilde{\tau}\cdot\partial_{y_1}[S(x_1,y_1)-S(w_1,y_1)]|\ge c 2^{-\ell}|x_1-w_1|.
\end{equation}
Because the phase of $K_{m,n,\tilde{m},n}$ is not linear in $y_1$, applying $\M_{y_1}^*$ many times will result in the derivative hitting both $\sigma$ and higher derivatives of the phase function. To handle this technicality we refer to a standard calculus result which has been committed to paper as an appendix in \cite{AnClPrSe2018}, and state an immediate consequence below. 

\begin{lem}\label{IBPlemma} 
    Suppose $\phi\in C^\infty(\rr^d)$, and define the operator $\M =\langle \tfrac{\nabla \phi}{|\nabla \phi|^2},\nabla \cdot\rangle$. Suppose $g\in C^\infty(\rr^d)$ and there exists $D>0$ such that for every derivative $\partial^j$ of order $j\in\nn$, $|\partial^j g|\le D^j$. Assume that there exists some $E>0$ such that $|\nabla\phi|\ge E$ and $|\partial^j\phi|\le C_j D^{j-1} E$ for $j\ge 2$. Then for every $N>0$,
    \begin{equation}\label{IBPBound}
        |\left(\M^*\right)^N(g)|\lesssim_{N,d} \left(\frac{D}{E}\right)^N.
    \end{equation}
\end{lem}

%

We can apply Lemma \ref{IBPlemma} to the operator $\M_{y_1}$ using \eqref{lowerboundy1} and the assumption that $|\partial_{y_1}^j\sigma|\le C_j 2^{\ell j},$ which follows from the support of $\chimn_n(y_1)$. It suffices to check that $|\tilde{\tau}\cdot\partial_{y_1}^{j}[S(x_1,y_1)-S(w_1,y_1)]|\le C_j 2^{\ell(j-2)}|x_1-w_1|$ for $j\ge 2$. 
By a Taylor approximation and the fact that $\gamma(t)$ is $C^\infty$ (though for our purposes $C^5$ suffices) we have
\begin{align}
   \label{highyderiv} \tilde{\tau}\cdot\partial_{y_1}^{j}[S(x_1,y_1)-S(w_1,y_1)]= &\tilde{\tau}\cdot S_{x_1y_1^j}(w_1,y_1)(x_1-w_1)+|x_1-w_1|^2R_2(x_1,y_1), 
\end{align}
where $R_2(x_1,y_1)$ is smooth and uniformly bounded; this clearly satisfies the desired estimate. Thus we conclude
\[
    |\left(\M^*_{y_1}\right)^N(\sigma)|\lesssim_N \left(\frac{2^{2\ell}}{2^k|x_1-w_1|}\right)^N.
\]
We integrate by parts with the help of $\M_{y_1}^N$ to obtain
\begin{align*}
    \left|K_{m,n,\tilde{m},n}(x_1,w_1,\tilde{\tau})\right|
    &\lesssim_N \int \left(\frac1{2^k2^{-2\ell}|x_1-w_1|}\right)^N |\chimn_{n}(y_1)\chimn_{m}(x_1)\chimn_{\tilde{m}}(w_1)| \, dy_1.
\end{align*}
Since $2^{k-2\ell}|x_1-w_1|\simeq|m-\tilde{m}|>0$, we have
\begin{align*}
   \left|K_{m,n,\tilde{m},n}(x_1,w_1,\tilde{\tau})\right|&\lesssim_N \int \left(\frac1{|m-\tilde{m}|}\right)^{N} |\chimn_{n}(y_1)\chimn_{m}(x_1)\chimn_{\tilde{m}}(w_1)| \, dy_1 \\
   &\lesssim_N 2^{-\ell}\left(\frac1{|m-\tilde{m}|}\right)^{N} \chimn_{m}(x_1)\chimn_{\tilde{m}}(w_1).
\end{align*}
Combining this with the above estimate for the case $m=\tilde{m}$ yields \eqref{HHstarkernel}.

To prove \eqref{HstarH} we note that the Schwartz kernel for $H^*_{m,n}H_{\tilde{m},\tilde{n}}$ is given by
\begin{equation*}
    \tilde{K}_{m,n,\tilde{m},\tilde{n}}(y_1,z_1,\tilde{\tau})=\int e^{i2^k\tilde{\tau}\cdot(S(x_1,y_1)-S(x_1,z_1))} \tilde{\sigma}(x_1,y_1,z_1,\tilde{\tau}) \, dx_1,
\end{equation*}
where 
\begin{align*}
    \tilde{\sigma}(x_1,y_1,z_1,\tilde{\tau})&= |\chi_1(|\tilde{\tau}|)\chi(x_1)|^2 \chi(y_1)a_{k,\ell}(y_1,2^k\tilde{\tau})\chimn_m(x_1)\chimn_n(y_1) \\
    & \qquad \times\overline{\chi}(z_1)\overline{a_{k,\ell}}(z_1,2^k\tilde{\tau})\overline{\chimn_{\tilde{m}}}(x_1)\overline{\chimn_{\tilde{n}}}(z_1).
\end{align*}
Then as in \eqref{disjoint1} the kernel vanishes if $|m-\tilde{m}|>2\C$. We prove the rest of \eqref{HstarH}, or equivalently that the kernel $\tilde{K}_{m,n,\tilde{m},\tilde{n}}$ vanishes under the assumptions $m=\tilde{m}$ and $|n-\tilde{n}|>2\C$. Recall that $|n-\tilde{n}|>2\C$ implies that $|y_1-z_1|>C \C 2^{-\ell}$. If $\ell<5$ and we take $\tilde{C}>C^{-1}2^5$ then this implies the kernel is zero. Assuming $\ell\ge 5$, we see since $y_1\in \supp\, a_{k,\ell}(y_1,2^k\tilde{\tau})$ and $z_1\in\supp\, a_{k,\ell}(z_1,2^k\tilde{\tau})$ that
\[
    |\tilde{\tau}_3||y_1-z_1|\le \big(|\tilde{\tau}_2-y_1\tilde{\tau}_3|+|\tilde{\tau}_2-z_1\tilde{\tau}_3|\big)\le C2^{-\ell},
\]
Thus if we can prove a lower bound on $|\tilde{\tau}_3|$ and take $\C$ large enough (again, independently of $k$ and $\ell$) then these supports are disjoint and the kernel is 0. Indeed, since $|y_1|\le 1$, $|\tilde{\tau}_2|+|\tilde{\tau}_3|\ge \tfrac12$, and $|\tilde{\tau}_2-y_1\tilde{\tau}_3|\le 2^{-\ell}$ this follows from the reverse triangle inequality as long as $\ell\ge 5$. Let $\iota(\tilde{\tau})=\sgn(\tilde{\tau}_2\tilde{\tau}_3)$ so that $|\tilde{\tau}_2+\iota \tilde{\tau}_3|=|\tilde{\tau}_2|+|\tilde{\tau}_3|$. Then we see that 
\[
    2^{-\ell+2}\ge |\tilde{\tau}_2-y_1\tilde{\tau}_3|=|\tilde{\tau}_2+\iota \tilde{\tau}_3-(\iota+y_1)\tilde{\tau}_3|\ge \Big| |\tilde{\tau}_2|+|\tilde{\tau}_3|-|\iota+y_1| |\tilde{\tau}_3|\Big|.
\]
This implies that 
\[
|\tilde{\tau}_3|\ge \tfrac{|\tilde{\tau}_2|+|\tilde{\tau}_3|-2^{-\ell+2}}{|\iota+y_1|}\ge 2^{-3}-2^{-\ell+1}>2^{-4}>0.
\]

\end{proof}

\section{Decoupling for the Cone}\label{DecouplingSection}


The idea of the proof of Proposition \ref{decoupling} is to decouple along the fibers of the singular variety $\mathcal{L}^{\tilde{\Phi}}$, which as we have seen in Section \ref{conormalbundle} are curved cones varying with the base point. For background on this approach, see \cite{GrSe94,AnClPrSe2018,PrSe19}. 

Let us consider the moment curve $\gamma(t)=(t,t^2,\tfrac16 t^3)$ from Section \ref{modelcase} as an example. We have
\[
\A_{k,\ell}f(x)=\chi(x_1)\iint e^{i\tau\cdot (x'-y'+S(x_1,y_1))} \chi(y_1)\chi_k(|\tau|)a_{k,\ell,\pm}(y_1,\tau) f(y) \, dy \, d\tau,
\]
where from \eqref{Sdefn} 
\begin{align*}
S^1(x_1,y_1)&=-(x_1-y_1)^2=-x_1^2+2x_1y_1-y_1^2 \\
S^2(x_1,y_1)&=y_1(x_1-y_1)^2+\tfrac13(x_1-y_1)^3=\tfrac13x_1^3+\tfrac23y_1^3-x_1y_1^2.
\end{align*}
From \eqref{Sigmax} the cone $\Sigma_x$ associated to the moment curve is given by 
\[
\Big\{\rho\Big(\mx{(x_1-y_1)^2 \\ y_1 \\ 1}\Big) \ : \ \rho\in\rr, y_1\in\supp \chi\Big\},
\]
Through a nonlinear change of variables $\tilde{x}'=x'+\Big(\mx{-x_1^2 \\ \tfrac13x_1^3}\Big)$ and $\tilde{y}'=y'+\Big(\mx{y_1^2 \\ \tfrac23y_1^3}\Big)$, we are able to freeze the cone $\Sigma_x$. Observe $\A_{k,\ell}$ transforms into
\[
\A_{k,\ell}f(\tilde{x})=\chi(x_1)\iint e^{i(\tau_2(\tilde{x}_2-\tilde{y}_2+2x_1y_1)+\tau_3(\tilde{x}_3-\tilde{y}_3-x_1y_1^2))} \chi_k(|\tau|)a_{k,\ell}(y_1,\tau)f(y) \, d\tilde{y} \, d\tau. 
\]
After these changes of variables, the fibers of the singular variety are given by 
\[
\tilde{\Sigma}_x=\{\xi\in\rr^3 \ : \ \xi=\rho(y_1^2,y,1), \rho\in \rr \},
\]
which no longer varies with $x$ as in the case in Proposition 3.1 of \cite{PrSe06}. There it was shown that the Fourier transform of the associated operator $\A_{k,\ell}$ is essentially supported in a $2^{k-2\ell}$ neighborhood of $\Sigma_x$. Thus we can apply decoupling down to plates adapted to that thin neighborhood. 

In general it cannot be hoped that we can apply just one nonlinear change of variables to fix $\Sigma_x$. However, if we cut up the support in $x_1$ into small intervals, somewhat ``freezing'' the variation of the cone $\Sigma_x$ with $x_1$, we can apply decoupling to each decomposed piece separately. We cannot expect to be able to decouple down to the $2^{k-2\ell}$ scale immediately because of the variation in $\Sigma_x$ even over this smaller interval of $x_1$, but by decoupling many times, changing variables at each step to further ``freeze'' the cone $\Sigma_x$ depending on the decoupling step, we can recover the same estimate as in the model case above, with a large constant depending on $\varepsilon$.

\subsection{The Decoupling Step}

Given $a,b\in\rr$ and $\delta_0,\varepsilon_1>0$, define the operator
\[
Tf(x)=2^{2k}\iint e^{i2^k\tau\cdot\tilde{\Phi}(x,y)} \sigma_{a,b}(x_1,y_1,\tau) f(y) \, d\tau  \, dy,
\]
where 
\[
\sigma_{a,b}(x_1,y_1,\tau)=\chi_1(|\tau|)a_{k,\ell}(y_1,2^k\tau)\chi(\varepsilon_1^{-1}|x_1-a|)\chi(\delta_0^{-1}|y_1-b|).
\]
Let $I=\supp\, \chi(\delta_0^{-1}|y_1-b|)$ be an interval of length $\delta_0$ containing $b$. Let $\{\nu\}$ be a set of $\delta_1$-separated points in $I$, and let $I_{\nu}$ be disjoint intervals containing $\nu$ so that $|I_{\nu}|=\delta_1$ and $I=\cup_{\nu} I_{\nu}$. Let $f_{\nu}(y)=f(y)\1_{I_\nu}(y_1)$. We first prove the following inductive step.

\begin{prop}\label{decouplingstep}
Let $2\le p\le 6$, $0<\varepsilon\le 1$, $k\gg 1$, $\ell\le k/2$, $2^{-\ell\varepsilon}>\delta_0> 2^{-\ell(1-\varepsilon)}$, and $\delta_1<\delta_0$ be given such that $\delta_1\ge \max\{2^{-\ell(1-\varepsilon/2)},\delta_0 2^{-\ell\varepsilon/4}\}$. Define $\varepsilon_1=(\delta_1/\delta_0)^2$ and fix $a,b\in\rr$. Then we have, for any $\varepsilon'\in(0,\varepsilon)$ and $N\in\nn$,
\[
\Big\|T\Big[\sum_{\nu} f_{\nu}\Big]\Big\|_{p}\lesssim_{\varepsilon'}(\delta_0/\delta_1)^{1/2-1/p+\varepsilon'}\left(\sum_{\nu} \Big\|T f_{\nu}\Big\|_{p}^p\right)^{1/p}+C(\varepsilon,N)2^{-kN}\sup_{\nu}\|f_{\nu}\|_p.
\]
\end{prop}

Since $y_1$ lies within $\delta_0$ of $b$ we change variables $x'\mapsto x'-S(x_1,b)$. Note that this change of variables is a diffeomorphism and the determinant of the Jacobian is 1. Under this change of variables our phase function for the operator $T$ becomes
\[
\Phi_b(x,y)=x'-y'+(S(x_1,y_1)-S(x_1,b))
\]
and the fibers $\Sigma_x^{\Phi_b}$ become 
\[
\Sigma_{x}^{\Phi_b}:=\left\{\rho\left(\mx{\gamma_2(x_1-y_1)-\gamma_2(x_1-b)+(y_1-b)\gamma_2'(x_1-b) \\ y_1 \\ 1}\right) \, : \, \rho\in\rr, |y_1-b|\le \delta_0 \right\}.
\]
A basis for the tangent space of $\Sigma_x^{\Phi_b}$ is given by the radial and nonradial vectors
\begin{align*}
u_1(x_1,y_1)&=\left(\mx{\gamma_2(x_1-y_1)-\gamma_2(x_1-b)+(y_1-b)\gamma_2'(x_1-b) \\ y_1 \\ 1}\right) \\
\tilde{u}_2(x_1,y_1)&=\left(\mx{\gamma_2'(x_1-b)-\gamma_2'(x_1-y_1) \\ 1 \\ 0}\right),
\end{align*}
and the normal vector is given by
\begin{align*}
u_3(x_1,y_1)&:=\tilde{u}_2(x_1,y_1)\times u_1(x_1,y_1) \\
&=\left(\mx{1 \\ S_{x_1}(x_1,b)-S_{x_1}(x_1,y_1)}\right).
\end{align*}
Define $u_2(a,\nu):=\tilde{u}_2(a,\nu)-e_3 \langle \tilde{u}_2(a,\nu),u_1(a,\nu)\rangle$ so that $\{u_i\}_{i=1}^2$ defines an orthogonal basis for the tangent space of $\Sigma_x^{\Phi_b}$ and let $\Pi_{a,\nu}(\delta_1)$ be the plate adapted to the cone $\Sigma_x^{\Phi_b}$ at the point $(a,\nu)$, defined by the inequalities
\begin{align}
& B^{-1}\le |\xi| \le B \label{verplatecomp} \\
& |\left\langle u_2(a,\nu), \xi\right\rangle|\le B\delta_1 \label{horplatecomp}\\
& |\langle u_3(a,\nu),\xi\rangle|\le B\delta_1^2. \label{normplatecomp}
\end{align}

Let $\chi_{\nu}$ be a smooth bump function equal to $1$ on $\Pi_{a,\nu}(\delta_1)$ and supported on its double. Let $P_{k,\nu}^a$ be the Fourier projection operator onto $2^k\Pi_{a,\nu}(\delta_1)$, defined by
\[
\reallywidehat{P_{k,\nu}^a f}(\xi):=\chi_{\nu}(2^{-k}\xi)\hat{f}(\xi).
\]
By the Bourgain-Demeter decoupling inequality on the cone \cite{BoDe15}, H\"older's inequality, and an application of our change of variables again, we have
\begin{equation}\label{decouplingplate}
\Big\|\sum_{\nu} P_{k,\nu}^a Tf_{\nu}\Big\|_{p}\lesssim_\varepsilon(\delta_0/\delta_1)^{1/2-1/p+\varepsilon}\left(\sum_{\nu} \|T f_{\nu}\|_{p}^p\right)^{1/p},
\end{equation}
for $2\le p\le 6$. See \cite{PrSe07} for an argument adapting this result (actually an older decoupling estimate for large $p$ proven by Wolff in \cite{Wo00}) to a general conic surface with one non-vanishing principal curvature. To finish the proof of Proposition \ref{decouplingstep} we need to show that the $L^p$ norm of $(Id-P_{k,\nu}^a)T f_\nu$ is ``negligible.'' We follow the argument found in \S 4.1 of \cite{AnClPrSe2018}. We write the Schwartz kernel of the map $f\mapsto (Id-P_{k,\nu}^a)Tf$ as $\sum_{n=0}^\infty K_{n,k,\ell}$, where
\begin{align}
\label{platekernel} K_{n,k,\ell}(\tilde{x},y)&=  2^{2k} \iint e^{i\langle \tilde{x}-x,\xi\rangle} (1-\chi_{\nu}(2^{-k}\xi))\chi_n(|\xi|) \int e^{i2^k\tau\cdot\Phi_b(x,y)} \sigma_{a,b}(x_1,y_1,\tau) \, d\tau \,dx \,d\xi. 
 \\
&= 2^{2k} \iiint e^{i\Psi(x,\tilde{x},y,\tau,\xi)} \tilde{\sigma}_{a,b}(x_1,y_1,\tau,\xi) \, d\tau \, dx\, d\xi. \notag
\end{align}
We will use three integration by parts estimates to show that $K_{n,k,\ell}$ decays rapidly in $k$ and in $n$. Note that $\tilde{\sigma}_{a,b}$ is supported where $|\xi|\simeq 2^n$ for $n\ge 1$ and $|\xi|\le 1$ for $n=0$, and in the complement of the set $2\Pi_{a,\nu}(\delta_1)$, with $|\tau|\simeq 1$, and with $x_1$ and $y_1$ in a compact set. Further, since $f$ is compactly supported in $y$, $x$ is also restricted to a compact set. 

First, note that $\nabla_{\xi}\Psi=\tilde{x}-x$ and for any $\xi$-derivative $\partial_{\xi}^j$ of order $j$, $| \partial_\xi^j\tilde{\sigma}_{a,b}|\le C_j \min\{2^n,2^k\delta_1^2\}^{-j}$, so by Lemma \ref{IBPlemma}, we see that
\begin{equation} \label{xiibp}
\Big|\int e^{i\Psi(x,\tilde{x},y,\tau,\xi)} \tilde{\sigma}_{a,b}(x_1,y_1,\tau,\xi) \, d\xi\Big|\le C_N \frac{2^n}{(1+|x-\tilde{x}|)^N}.
\end{equation}
We will use this estimate to allow us to integrate in $x$ in the next two integration by parts estimates. 

Assume that $|k-n|>C$ for some positive constant $C>1$. Then we integrate by parts in the $x$-variables. Note that since $\nabla_x\Phi^i_b$ are linearly independent for $i=1,2$ and roughly size 1, we have the lower bound
\[
|\nabla_x\Psi|=|-\xi+2^k(\tau\cdot \Phi_b(x,y))_x|\ge C\max\{2^k,2^n\}
\]
On the other hand, higher $x$-derivatives of $\Psi$ are bounded by $C_j 2^k$, and higher $x$-derivatives of $\tilde{\sigma}_{a,b}$ are bounded by $C_j 2^{\ell j}$. Since $\ell\le k/2$, integration by parts in the $x$-variables, followed by the $\xi$-variables, by two applications of Lemma \ref{IBPlemma} yields
\[
|K_{n,k,\ell}(\tilde{x},y,\xi)|\le C_N \iiint \frac{1}{(\max\{2^k,2^n\})^N(1+|x-\tilde{x}|)^N} \, dx \, d\tau \, d\xi,
\]
which decays rapidly in $n$ for $n>k$, and rapidly in $k$ for $n<k$. Thus we have only to show that similar estimates hold when $n\simeq k$.

Assume that $|n-k|\le C$. We will first integrate by parts in the $x$-variables using the fact that $\xi$ lie on the complement of the plates $\Pi_{a,\nu}(\delta_1)$, then our initial estimate integrating by parts in the $\xi$-variables. To use integration by parts estimates on the complements of the plates $\Pi_{a,\nu}(\delta_1)$ we first formulate upper bounds on certain directional derivatives of $\tau\cdot\Phi_b$ which will become lower bounds for derivatives of $\Psi$. Notice that our assumption $|n-k|\le C$ above ensures that \eqref{verplatecomp} holds, hence we only need to consider $\xi$ for which either \eqref{horplatecomp} or \eqref{normplatecomp} do not hold.

\begin{lem}\label{plateestimates}
There is a constant $A\ge 1$ so that for all $|y_1-\nu|\le \delta_1$,
\begin{align}
|\langle u_2(a,\nu),\nabla_x (\tau\cdot\Phi_b)\rangle|&\le A\delta_1 \label{tangent} \\
|\langle u_3(a,\nu),\nabla_x (\tau\cdot\Phi_b)\rangle|&\le A\delta_1^2. \label{normal}
\end{align}
\end{lem}

\begin{proof}

We begin with the more delicate inequality, \eqref{normal}. Taking a Taylor expansion of the left hand side about $(x_1,y_1)=(a,\nu)$ we see that for sufficiently large $K\lesssim 1/\varepsilon$ we have
\begin{align}
\label{normaltaylorexp} \langle u_3(a,\nu),\nabla_x(\tau\cdot\Phi_b)\rangle &=  \tau\cdot\left(S_{x_1}(x_1,y_1)-S_{x_1}(x_1,b)-S_{x_1}(a,\nu)+S_{x_1}(a,b)\right) \\
&= 0+ (y_1-\nu)\tau\cdot S_{x_1 y_1}(a,\nu) +\tfrac{1}2(y_1-\nu)^2\tau\cdot S_{x_1 y_1^2}(a,\nu) \notag\\
& \qquad   +\mathcal{I}+\mathcal{I}\mathcal{I}+\mathcal{I}\mathcal{I}\mathcal{I}+\delta_1^2R_3(x_1,y_1) \notag,
\end{align}
where 
\begin{align*}
\mathcal{I}&= \sum_{j=2}^{K+1} \tfrac{(x_1-a)^{j-1}}{(j-1)!}\tau\cdot[S_{x_1^j}(a,\nu)-S_{x_1^j}(a,b)] \\
\mathcal{I}\mathcal{I}&=\sum_{j=2}^{K} \tfrac{(x_1-a)^{j-1}(y_1-\nu)}{(j-1)!}\tau\cdot  S_{x_1^{j}y_1}(a,\nu)  \\
\mathcal{I}\mathcal{I}\mathcal{I}&=\sum_{j=2}^{K-1} \tfrac{(x_1-a)^{j-1}(y_1-\nu)^2}{2(j-1)!} \tau\cdot  S_{x_1^{j}y_1^2}(a,\nu),
\end{align*}
and $R_3$ is smooth and uniformly bounded. 
The first two nonzero terms of \eqref{normaltaylorexp} represent respectively the size of $\det d\pi_L^{\Phi_b}$ and $V_L^{\Phi_b}\det d\pi_L^{\Phi_b}$ near the singular variety $\mathcal{L}^{\Phi_b}$. From \eqref{Sxjy}, \eqref{Sxjy2}, and the fact that $|y_1-\nu|\le \delta_1$ we see that 
$\left|\tau\cdot S_{x_1 y_1}(a,\nu)\right|
\le  C\delta_1 
$
and
$
\left|\tau\cdot S_{x_1 y_1^2 }(a,\nu)\right|
\simeq 1.
$
Note that in the model case of the moment curve (corresponding to $\gamma_2(t)=t^2$) the terms $\mathcal{I}$, $\mathcal{I}\mathcal{I}$, and $\mathcal{I}\mathcal{I}\mathcal{I}$ vanish identically, and the bound on the normal directional derivative is a simple geometric statement about the size of $\det d\pi_L^{\Phi_b}$ and the non-vanishing of $V_L\det d\pi_L^{\Phi_b}$. For more general curves we have to estimate $\mathcal{I}, \mathcal{I}\mathcal{I},$ and $\mathcal{I}\mathcal{I}\mathcal{I}$, which encode the behavior of higher order mixed kernel fields acting on the determinant. Thankfully, these are easy to estimate by our calculations in Section \ref{conormalbundle}.
 
By \eqref{Sxjy}, \eqref{Sxjy2}, and the fact that $|y_1-\nu|\le \delta_1$ we see that 
\begin{align*}
\mathcal{I}\mathcal{I}+\mathcal{I}\mathcal{I}\mathcal{I}\lesssim_K\varepsilon_1\delta_1^2.
\end{align*}

To analyze $\mathcal{I}$ we consider a Taylor expansion of $\tau\cdot\big[S_{x_1^j}(a,y_1)-S(a,b)\big]$ about $y_1=b$. We see that by \eqref{Sxjy}, \eqref{Sxjy2}, and the fact that $|y_1-b|\le \delta_0$ we have
\begin{align}\label{taylorSxj}
\Big|\tau\cdot\left[S_{x_1^j}(a,y_1)-S_{x_1^j}(a,b)\right]\Big|&\le \Big|(y_1-b)\tau\cdot S_{x_1^jy_1}(a,b)\Big| +\Big|\tfrac{(y_1-b)^2}{2}\tau\cdot  S_{x_1^j y_1^2}(a,b)\Big|+|y_1-b|^3R_4(y_1)  \\
&\le C\delta_0^2, \notag
\end{align}
where $R_4(y_1)$ is smooth and uniformly bounded. Hence
\[
\mathcal{I}\lesssim\varepsilon_1\delta_0^2,
\]
and so combining our estimates for $\mathcal{I},\mathcal{II}$, and $\mathcal{III}$, and using the fact that $\varepsilon_1=(\delta_1/\delta_0)^2$ we have
\[
\langle u_3(a,\nu),\nabla_x(\tau\cdot\Phi_b)\rangle\le C_K(\delta_1^2+\varepsilon_1\delta_0^2+\varepsilon_1\delta_1^2)\le C_K\delta_1^2,
\]
as desired.

We next prove \eqref{tangent}. Expanding the left hand side about $(x_1,y_1)=(a,\nu)$ we see that
\begin{align*}
\langle u_2(a,\nu),\nabla_x(\tau\cdot\Phi_b)\rangle &= (\tau_2-\nu\tau_3) \\
& \qquad + \left(\gamma_2'(a-b)-\gamma_2'(a-\nu)\right)\Big[\tau\cdot\left(S_{x_1}(x_1,y_1)-S_{x_1}(x_1,b)\right) \\
& \qquad - \tau_3\left(\gamma_2(a-\nu)-\gamma_2(a-b)+(\nu-b)\gamma_2'(a-b)\right)\Big] \\
&= (\tau_2-\nu\tau_3)\left[1+\left(\gamma_2'(a-\nu)-\gamma_2'(a-b)\right)^2\right] \\
& \qquad + (\gamma_2'(a-\nu)-\gamma_2'(a-b)) F(x_1,y_1),
\end{align*}
where 
\begin{align*}
F(x_1,y_1)&= (x_1-a)\left[\tau\cdot\left(S_{x_1^2}(a,\nu)-S_{x_1^2}(a,b)\right)\right] \\
& \qquad + (y_1-\nu)\left[\tau\cdot S_{x_1 y_1}(a,\nu)\right] \\
& \qquad + \varepsilon_1^2R_5(x_1,y_1),
\end{align*}
where $R_5$ is smooth and uniformly bounded. Using \eqref{Sxjy} and an argument similar to \eqref{taylorSxj}, along with the estimates $|x_1-a|\le \varepsilon_1$, $|y_1-\nu|\le \delta_1$, and $|\nu-b|\le \delta_0$ we see
\[
F(x_1,y_1)\lesssim \varepsilon_1\delta_0^2+\delta_1^2+\varepsilon_1^2.
\]
Thus we conclude
\begin{align*}
\langle u_2(a,\nu),\nabla_x(\tau\cdot\Phi_b)\rangle&\lesssim (2^{-\ell}+\delta_1)[1+\delta_0^2]+\delta_0(\varepsilon_1\delta_0^2+\delta_1^2+\varepsilon_1^2) \\
&\lesssim \delta_1,
\end{align*}
as desired.

\end{proof}

With Lemma \ref{plateestimates} proven, we can use integration by parts to estimate the kernel of $(Id-P_{k,\nu}^a(\delta_1))T$ in the complement of the plates. Assume first that \eqref{horplatecomp} does not hold, i.e. 
\[
|\langle u_2(a,\nu),\xi\rangle|\ge B\delta_1, 
\]
with $B\ge 2A\ge 2$. Then from inequality \eqref{tangent} in Lemma \ref{plateestimates} we get
\[
|\langle u_2(a,\nu),\nabla_x\Psi\rangle\ge (B-A)\delta_1\ge \delta_1.
\]
Define a differential operator $L$ by
\[
Lh=\left\langle u_2(a,\nu),\nabla_x\left(\frac{h}{|\langle u_2(a,\nu),\nabla_x\Psi\rangle|}\right)\right\rangle.
\]
Then by integration by parts, the kernel from \eqref{platekernel} becomes 
\[
K(\tilde{x},y,\xi)=i^N2^{-Nk}\iiint e^{i 2^k\Psi(x,\tilde{x},y,\tau,\xi)}L^N(\sigma(x,y,\tau,\xi)) \, d\tau \,dx \,d\xi.
\]
To estimate $|L^N(\sigma)|$ we use Lemma \ref{IBPlemma}. We see that $|\langle u_2(a,\nu),\nabla_x\rangle^j \sigma|\le \varepsilon_1^{-j}$ due to the support of $x_1$, so to apply the lemma we just need to check that 
\[
|\langle u_2(a,\nu),\nabla_x\rangle^j \Psi|\le C_j \varepsilon_1^{1-j} \delta_1=C_j\delta_0^{2j-2}\delta_1^{3-2j}, \qquad j\ge 2.
\]
Indeed, since $\Psi$ is linear in $x'$,
\begin{align*}
\left|\langle u_2(a,\nu),\nabla_x\rangle^j \Psi\right|&=\left|(S_{x_1}(a,\nu)-S_{x_1}(a,b))^j\partial_{x_1}^j\tau\cdot \left[ S(x_1,y_1)-S(x_1,b)\right]\right|  \\
&\le C_j\delta_0^{j+2} \\
&\le C_j\varepsilon_1^{1-j}\delta_1,
\end{align*}
using the estimate in \eqref{taylorSxj} and the fact that $\delta_1\le \delta_0\le 1$. Thus $|L^N\sigma|\lesssim_N(\varepsilon_1\delta_1)^{-N}$ by Lemma \ref{IBPlemma} and integration by parts gains a factor of
\[
C_N \left(2^{k}\varepsilon_1\delta_1\right)^{-N}\lesssim_N (2^{k-\ell})^{-N}\lesssim_N2^{-Nk/2},
\]
since $\varepsilon_1=\delta_1^2\delta_0^{-2}\ge 2^{-\ell\varepsilon/2}\ge 2^{-\ell\varepsilon}$, $\delta_1\ge 2^{-\ell(1-\varepsilon)}$, and $\ell\le k/2$. Integrating by parts in the $\xi$-variables and using the estimate obtained in \eqref{xiibp} yields the desired estimate.

Next we consider the case that \eqref{normplatecomp} does not hold, i.e. $|\langle u_3(a,\nu),\xi\rangle|\ge B\delta_1^2,$ for some $B\ge 2A$. Then by Lemma \ref{plateestimates}
\begin{equation}
|\langle u_3(a,\nu),\nabla_x \Psi\rangle|\ge (B-A)\delta_1^2\ge \delta_1^2. \label{normlowerbound}
\end{equation}
Define the differential operator $\tilde{L}$ to be
\[
\tilde{L}h=\left\langle u_3(a,\nu),\nabla_x\left(\frac{h}{|\langle u_3(a,\nu),\nabla_x\Psi\rangle|}\right)\right\rangle.
\]
Again by integration by parts the kernel in \eqref{platekernel} becomes
\[
K(\tilde{x},y,\xi)=i^N2^{-Nk}\iiint e^{i 2^k\Psi(x,\tilde{x},y,\tau,\xi)}\tilde{L}^N(\sigma(x,y,\tau,\xi)) \, d\tau \,dx \,d\xi.
\]
Again by Lemma \ref{IBPlemma} and the lower bound \eqref{normlowerbound} it suffices to check
\[
|\langle u_3(a,\nu),\nabla_x\rangle^j\Psi|\le C_j\varepsilon_1^{1-j}\delta_1^2=C_j\delta_0^2(\delta_0/\delta_1)^{2j-4}, \qquad j\ge 2.
\]
The linearity of $\Psi$ in $x'$ saves us yet again, as
\begin{align*}
|\langle u_3(a,\nu),\nabla_x\rangle^j\Psi|&=\partial_{x_1}^j\tau\cdot\left[S(x_1,y_1)-S(a,b)\right]\le C_j\delta_0^2 \le C_j\delta_0^2(\delta_0/\delta_1)^{2j-4},
\end{align*}
by a calculation and another application of \eqref{taylorSxj}. Thus $|\tilde{L}^N\sigma|\lesssim_N(\varepsilon_1\delta_1^2)^{-N}$ by Lemma \ref{IBPlemma}, and integration by parts gains a factor of 
\[
C_N(2^k\varepsilon_1\delta_1^2)^{-N}\lesssim_N(2^{k-2\ell(1-3\varepsilon/4)})^{-N}\lesssim_N 2^{-\tfrac34 k\varepsilon N}.
\]
Thus for each $\varepsilon>0$ and $N>0$ we can pick $N_1>0$ such that applying integration by parts $N_1$ times with respect to $\tilde{L}$ gains a factor of
\[
C(N,\varepsilon)2^{-Nk}.
\]
Repeated integration by parts in the $\xi$ variables using the estimate obtained by \eqref{xiibp} yields the desired estimate. 

To finish the proof of Proposition \ref{decouplingstep} we combine the decoupling estimate \eqref{decouplingplate} and the above analysis of the error to get the bound
\[
\Big\|T\Big[\sum_{\nu} f_{\nu}\Big]\Big\|_{p}\lesssim_{\varepsilon'}(\delta_0/\delta_1)^{1/2-1/p+\varepsilon'}\Big(\sum_{\nu} \Big\|T f_{\nu}\Big\|_{p}^p\Big)^{1/p}+C(\varepsilon,N)2^{-kN}\sup_{\nu}\|f_{\nu}\|_p.
\]

\subsection{Iteration of the Decoupling Step}

Let $\delta_0=2^{-\ell\varepsilon/8}$, and define $\delta_j=\delta_{j-1}2^{-\ell\varepsilon/4}$ for $j=1,2,...$ Note that this implies $\varepsilon_1=(\delta_1/\delta_0)^2=2^{-\ell\varepsilon/2}$. We will iterate the estimate in Proposition \ref{decouplingstep} until $\delta_j\le 2^{-\ell(1-\varepsilon)}$. Let $j^*$ be the smallest $j$ such that $\delta_j<2^{-\ell(1-\varepsilon)}$. Clearly $j^*\lesssim 1/\varepsilon$ and $2^{-\ell(1-\varepsilon/2)}\le \delta_{j^*}\le 2^{-\ell(1-\varepsilon)}$. Pick a lattice of points $\mathcal{Z}\subset\rr$ such that $|a-a'|=2^{-\ell\varepsilon/2}$ for $a,a'\in\mathcal{Z}$, $a\ne a'$. Note that
\[
\sum_{a\in\mathcal{Z}} \chi(\varepsilon_1^{-1}(\cdot-a))\equiv 1.
\]
Using this partition of unity we decompose our operator
\[
\A_{k,\ell} f(x)=\sum_{a\in\Z} \chi(\varepsilon_1^{-1}(x_1-a))\A_{k,\ell} f(x)=:\sum_{a\in\Z} \A_{k,\ell}^a f(x).
\]
For each $j\in\{0,1,...,j^*\}$ let $\{I_{\nu_j}\}$ denote the collection of disjoint dyadic intervals of length $\delta_j$ tiling $I$. Let $f_{\nu_j}(y_1,y')=f(y_1,y')\1_{I_{\nu_j}}(y_1)$. Then by Minkowski's and H\"older's inequalities
\begin{align}\label{basecase} 
\|\A_{k,\ell} f\|_p & \lesssim\sum_{\nu_0}\Big(\sum_{a\in\Z} \|\A_{k,\ell}^a f_{\nu_0}\|_p^p\Big)^{1/p}  \\
& \lesssim 2^{\ell\varepsilon/8p'}\Big(\sum_{\nu_0} \sum_{a\in\Z} \|\A_{k,\ell}^a f_{\nu_0}\|_p^p\Big)^{1/p}. \notag
\end{align}

The operator $\A_{k,\ell}^a$ and the function $f_{\nu_0}$ now satisfy the conditions of Proposition \ref{decouplingstep}. We claim that for each $j\le j^*$,
\begin{align} \label{induction}
\|\A_{k,\ell} f\|_{p}&\lesssim C(\varepsilon')^j  2^{\ell\varepsilon/(8p')} (\delta_0/\delta_j)^{1/2-1/p+\varepsilon'}\Big(\sum_{a\in\Z} \sum_{\nu_j} \|\A_{k,\ell}^a f_{\nu_j}\|_p^p\Big)^{1/p} \\
& \qquad  + j 2^{2\ell} C(\varepsilon')^{j-1} C(\varepsilon,N) 2^{-kN}\|f\|_p.\notag
\end{align}
The case $j=0$ follows immediately from \eqref{basecase}. Assume \eqref{induction} holds for some $j$. Then by applying Proposition \ref{decouplingstep} and the fact that $\ell^p\subset \ell^\infty$ we get
\begin{align}
\Big(\sum_{a\in\Z} \sum_{\nu_{j}}\| \A_{k,\ell}^a f_{\nu_j}\|_{p}^p\Big)^{1/p}&\le \Big(\sum_{a\in \Z} \sum_{\nu_j} \Big[C(\varepsilon')\big(\tfrac{\delta_j}{\delta_{j+1}}\big)^{1/2-1/p+\varepsilon'}\Big(\sum_{\nu_{j+1}\in I_{\nu_{j}}} \|\A_{k,\ell} f_{\nu_{j+1}}\|_p^p\Big)^{1/p} \notag \\
& \qquad + C(\varepsilon,N) 2^{-kN}\Big(\sum_{\nu_{j+1}\in I_{\nu_j}} \|f_{\nu_{j+1}}\|_p^p\Big)^{1/p}\Big]^p\Big)^{1/p}\label{applydecoupling} \\
&\le  C(\varepsilon')\big(\tfrac{\delta_j}{\delta_{j+1}}\big)^{1/2-1/p+\varepsilon'}\Big(\sum_{a\in\Z}\sum_{\nu_{j+1}}\|\A_{k,\ell}^a f_{\nu_{j+1}}\|_p^p\Big)^{1/p} \notag \\
& \qquad + C(\varepsilon,N)2^{\ell\varepsilon/2p}2^{-kN}\|f\|_p. \notag
\end{align}
Plugging the above estimate into \eqref{induction} gives us
\begin{align*}
\|\A_{k,\ell} f\|_p &\le C(\varepsilon')^{j+1}2^{\ell\varepsilon/(8p')}\big(\tfrac{\delta_{0}}{\delta_{j+1}}\big)^{1/2-1/p+\varepsilon'}\Big(\sum_{a\in\Z}\sum_{\nu_{j+1}} \|\A_{k,\ell}^a f_{\nu_{j+1}}\|_p^p\Big)^{1/p} \\ 
& \qquad + C(\varepsilon')^j2^{\ell\varepsilon/(8p')}\big(\tfrac{\delta_0}{\delta_j}\big)^{1/2-1/p+\varepsilon'}C(\varepsilon,N)2^{-kN}\|f\|_p \\
& \qquad + j2^{2\ell}C(\varepsilon')^{j-1}C(\varepsilon,N)2^{-kN}\|f\|_p.
\end{align*}
Using the fact that $\delta_0=2^{-\ell\varepsilon/8}$, $\delta_{j}\ge 2^{\ell(1-\varepsilon/2)}$ for $j\le j^*$, and $2\le p\le 6$, the last two terms of the above inequality are bounded by
\[
(j+1)C(\varepsilon')^j2^{2\ell}C(\varepsilon,N) 2^{-kN}\|f\|_p,
\]
proving the claim. 

We apply \eqref{induction} for $j=j^*$ and use the fact that $j^*\le 4/\varepsilon$ to deduce
\begin{align}
\|\A_{k,\ell} f\|_p & \le  C(\varepsilon')^{4/\varepsilon}2^{\ell\varepsilon/8p'}2^{-\ell\varepsilon/8(1/2-1/p+\varepsilon')}2^{\ell(1-\varepsilon/2)(1/2-1/p+\varepsilon')}\\
& \qquad \times\Big(\sum_{\nu_{j^*}}\|\A_{k,\ell} f_{\nu_{j^*}}\|_p^p\Big)^{1/p} \notag \\
& \qquad + \tfrac{4}{\varepsilon}C(\varepsilon')^{4/\varepsilon}C(\varepsilon,N) 2^{-kN+2\ell}\|f\|_p  \notag \\
& \lesssim_{\varepsilon,\varepsilon'} 2^{\ell\varepsilon/8}2^{\ell(1-\varepsilon/2)(1/2-1/p+\varepsilon')}\Big(\sum_{\nu_{j^*}} \|\A_{k,\ell}f_{\nu_{j^*}}\|_p^p\Big)^{1/p} \notag \\
& \qquad + C(\varepsilon,N_1)2^{-kN_1}\|f\|_p.  \notag 
\end{align}

\section{Bounds in Sobolev Spaces}\label{PRSSection}

We prove Theorem \ref{mainthm} from Proposition \ref{decomposedfinal}. Here we refer to a Calderon-Zygmund estimate found in \cite{PrRoSe11} which we will apply to $\A_{k,\ell}$ for fixed $\ell$. Let
\[
\A_\ell=\sum_{\substack{k\ge 2\ell}} \A_{k,\ell}.
\]
We will prove for compactly supported $f$
\[
\|\A_{\ell}f\|_{F^{p,q}_{1/p}}\le 2^{-\ell\varepsilon(p)}\|f\|_{B^{p,p}_{0}}, \qquad 0< q\le 2<4<p<\infty,
\]
where $F_{s}^{p,q}$ and $B_s^{p,q}$ are respectively the Triebel-Lizorkin space and Besov spaces (see \cite{Tr83}).
Summing in $\ell$ with $q\ge 1$ we conclude that
\[
\A: B^{p,p}_{s,comp}\to F^{p,q}_{s+1/p}, \qquad q\le 2<4<p<\infty.
\]
Since $L^p_s=F^{p,2}_s\xhookrightarrow{} B^{p,p}_s$ for $p>2$ and $F^{p,q}_{s+1/p}\xhookrightarrow{} F^{p,2}_{s+1/p}=L^p_{s+1/p}$ for $q\le 2$, this implies the asserted $L^p$-Sobolev bounds for $\A$ and by a change of variables, the bounds for $A$. 

Let $P_k$ be standard Littlewood-Paley multipliers on $\rr^3$ for $k\in\nn$. Because $\nabla_{x}\tilde{\Phi}_j(x,y)$ are linearly independent, as are 
$\nabla_{y}\tilde{\Phi}_j(x,y)$, we can find $C_0>0$ such that 
\[
4C_0^{-1}|\tau|\le |(\tau\cdot \tilde{\Phi})_x|,|(\tau\cdot\tilde{\Phi})_y|\le C_0/4 |\tau|
\]
This implies the following.
\begin{lem}\label{LittlewoodPaley} Suppose $k',k''\in \nn$, $k'\ge 2\ell$ and $\max\{|k-k'|,|k-k''|\}\ge C_1$, where $C_1$ depends on $C_0$. Then 
\[
\|P_k \A_{k',\ell} P_{k''}\|_{L^p\to L^p}\le C \min\{2^{-kN},2^{-k'N},2^{-k''N}\}.
\]
\end{lem}

\bp[Proof of Lemma \ref{LittlewoodPaley}] We follow a similar argument to that laid out in \cite{Se93}. Note that the kernel of the operator $P_k \A_{k',\ell} P_{k''}$ is given by
\begin{align*}
\int\int\int\int\int e^{i\left[\langle x-w,\eta\rangle+\tau\cdot\tilde{\Phi}(w,z)+\langle z-y,\xi\rangle\right]} \chi_1(2^{1-k}|\eta|)\chi_1(2^{1-k'}|\tau|)\chi_1(2^{1-k''}|\xi|) \\
\qquad \times a_{k,\ell,\pm}(z_1,\tau) \chi(|w|)\chi(|x|)\, dw \,dz \,d\tau \,d\eta \,d\xi.
\end{align*}
Our assumption on $\Phi$ implies that if $\max\{|k-k'|,|k'-k''|\}>C_1$ we have
\[
\nabla_{(z,w)}\left[ \langle x-w,\eta\rangle+\tau\cdot\tilde{\Phi}(w,z)+\langle z-y,\xi\rangle\right]\ge c \max\{2^k,2^{k'},2^{k''}\}.
\]
Thus we integrate by parts in the $(w,z)$ variables to get the above bound on the kernel, implying by Minkowski the desired bound on $L^p$. 

\ep

Using the lemma above and an argument similar to a part of the proof of Lemma 2.1 in \cite{Se93}, we can reduce the proof of Theorem \ref{mainthm} to the estimate
\begin{equation}\label{PRSinequality}
    \Big\|\Big(\sum_{k\ge 2\ell} \big|2^{k/p}P_k\A_{k+s_1,\ell}P_{k+s_2} f\big|^q\Big)^{1/q}\Big\|_{L^p}\le C 2^{-\ell\varepsilon(p)}\Big\|\Big(\sum_{k>0} |P_{k+s_2}f|^p\Big)^{1/p}\Big\|_{L^p}.
\end{equation}

To prove \eqref{PRSinequality} we apply the main result from \cite{PrRoSe11}.

\begin{thm}[\cite{PrRoSe11}]\label{PRS}
Let $T_k$ be a family of operators on Schwartz functions by
\[
    T_k f(x)=\int K_k(x,y) f(y) \, dy.
\]
Let $\phi\in \mathcal{S}(\rr^3)$, $\phi_k=2^{3k}\phi(2^k\cdot)$, and $\Pi_k f=\phi_k*f$. Let $\varepsilon>0$ and $1<p_0<p<\infty$. Assume $T_k$ satisfies
\begin{align}\label{PRSLp}
    \sup_{k>0} 2^{k/p}\|T_k\|_{L^p\to L^p}&\le A \\
    \sup_{k>0} 2^{k/p_0}\|T_k\|_{L^{p_0}\to L^{p_0}}&\le B_0. \label{PRSLpo}
\end{align}
Further let $A_0 \ge 1$, and assume that for each cube $Q$ there is a measurable set $E_Q$ such that
\begin{equation}\label{EQbound}
    |E_Q|\le A_0\max\{|Q|^{2/3},|Q|\},  
\end{equation}
and for every $k\in\nn$ and every cube $Q$ with $2^k {\rm diam }(Q)\ge 1,$
\begin{equation}\label{EQintegral}
    \sup_{x\in Q} \int_{\rr^d\setminus E_Q} |K_k(x,y)| \, dy\le B_1 \max\left\{\left(2^k{\rm diam}(Q)\right)^{-\varepsilon},2^{-k\varepsilon}\right\}.
\end{equation}
Let 
\[
\mathcal{B}=B_0^{q/p}(A A_0^{1/p}+B_1)^{1-q/p}.
\]
Then for any $q>0$ there is a $C$ depending on $\varepsilon,p,p_0,q$ such that
\begin{equation}\label{PRSestimate}
    \Big\|\Big(\sum_k 2^{kq/p}|P_k T_k f_k|^q\Big)^{1/q}\Big\|_{p}\le C A \left[\log\left(3+\tfrac{\mathcal{B}}{ A}\right)\right]^{1/q-1/p}\Big(\sum_k \|f_k\|_p^p\Big)^{1/p}.
\end{equation}
\end{thm}

We apply this theorem on the family of operators $T_k:=\A_{k,\ell}$ for $k\ge 2\ell$ (here $\ell$ is fixed). By Proposition \ref{decomposedfinal} the assumptions \eqref{PRSLp} and \eqref{PRSLpo} are satisfied with $A\lesssim 2^{-\ell\varepsilon(p)}$ and $B_0\lesssim 2^{-\ell\varepsilon(p_0)}$. We next check assumptions \eqref{EQbound} and \eqref{EQintegral}. For a given cube $Q$ with center $x_Q$ let
\[
E_Q=\{y \, : \, |(x')^Q-y'+S(x_1^Q,y_1)|\le C 2^{\ell}\diam(Q)\}
\]
if $\diam(Q)<1$, and a cube centered at $x^Q$ of diameter $C2^{\ell}\diam(Q)$ if $|Q|\ge 1$. By an integration by parts argument we derive the bound
\[
|K_k(x,y)|\lesssim_N\frac{2^{2k}}{(1+2^{k-\ell}|(x')^Q-y'+S(x_1^Q,y_1)|)^N}.
\]
Then clearly assumptions \eqref{EQbound} and \eqref{EQintegral} are satisfied with $A_0\lesssim 2^{3\ell}$ and $B_1\lesssim 2^{2\ell}$ respectively. Theorem \ref{PRS} then implies \eqref{PRSinequality} with $\Pi_k=P_{k+s_1}$ and $f_k=P_{k+s_2} f$, finishing the proof of Theorem \ref{mainthm}.

\subsection{Application to Sobolev Spaces Adapted to Heisenberg Translation} \label{HeisenbergSobolev}

Here we discuss the properties of an analogue of the Euclidean Sobolev spaces, given by the norm
\[
\|f\|_{L^p_s(\mathbb{H}^1)}=\Big\|\sum_{\lambda\in \Lambda} \mathcal{R}_\lambda D^s \psi_0 \mathcal{R}_{\lambda^{-1}} f\Big\|_{L^p}.
\]
This norm is a somewhat natural choice for a Sobolev space on $\mathbb{H}^1$ for three reasons. First, the Euclidean Sobolev norm and the Heisenberg-Sobolev norm are comparable for functions supported near the origin. Second, if we replace Heisenberg translations over $\Lambda$ with Euclidean translations over the integers (denote these translations $\tau_n$) we see that
\begin{align*}
    \Big\|\sum_{n\in\zz} \tau_n D^s \psi_0 \tau_{-n} f\Big\|_p&=\Big\|\sum_{n\in\zz} \tau_n D^s \tau_{-n}\psi_{n} f\Big\|_p \\
    &=\Big\|D^s \sum_{n\in\zz} \psi_n f\Big\|_p \\
    &=\|D^s f\|_p,
\end{align*}
assuming that $\sum_{n\in\zz} \psi_n\equiv 1$. So the only obstruction between this space and the standard (Euclidean) Sobolev space is the fact that $D^s$ does not commute with Heisenberg translations, making it a natural analogue of the Sobolev space. Third, this norm is independent of our choice of smooth cutoff function $\psi$. We prove this in Appendix \ref{triebelheisenbergappendix}

\subsubsection{Proof of Theorem \ref{heisensobocor}}

We use almost disjoint support of $\psi_\lambda$ and the fact that $A$ commutes with Heisenberg translation to show
\[
\|A f\|_{L^p_{1/p}(\mathbb{H}^1)}\lesssim \Big(\sum_{\lambda\in\Lambda} \|\mathcal{R}_\lambda D^{1/p} \psi_0 A \mathcal{R}_{\lambda^{-1}} f\|_p^p\Big)^{1/p}.
\]
We first remove the right translation by $\lambda$ by an affine change of variables. We observe from \eqref{Agamma} that for $\mathcal{F}$ a fixed dilate of the support of $\psi_0$ we have $\psi_0 A \mathcal{R}_{\lambda^{-1}} f=\psi_0 A \1_{\mathcal{F}}\mathcal{R}_{\lambda^{-1}}f$. This combined with Theorem \ref{mainthm} gives
\begin{align*}
    \Big(\sum_{\lambda\in\Lambda} \|D^{1/p} \psi_0 A \mathcal{R}_{\lambda^{-1}} f\|_p^p\Big)^{1/p}&\lesssim \Big(\sum_{\lambda\in\Lambda} \|\1_{\mathcal{F}}\mathcal{R}_{\lambda^{-1}} f\|_p^p\Big)^{1/p} \\
    &\lesssim \|f\|_p,
\end{align*}
finishing the proof.



\subsection{The proof of Corollary \ref{corollary}}

Theorem \ref{mainthm} establishes the estimates for $p>4$. For $2\le p<4$ the estimates in Corollary \ref{corollary} follow from interpolation with the $L^2$ regularity estimates of \cite{GrSe94}. As $A$ is clearly bounded on $L^1$, all that remains is to prove estimates for $1<p<2$. We prove a more general estimate that can apply to the fold and finite type conditions introduced in Section \ref{background}. 
\begin{prop}\label{Hardyprop}
        Suppose $\gamma''\ne 0$. Assume there exists $\alpha>0$ such that for $f\in L^2(\rr^3)$ with compact support,
        \begin{equation}\label{L2Hardy}
            \|Af\|_{L^2_\alpha(\rr^3)}\le C \|f\|_{L^2(\rr^3)}. 
        \end{equation}
        Then for $1<p<2$, $A$ is bounded from $L^p_{comp}(\rr^3)$ to $L^p_{\alpha(p)}(\rr^3)$ where $\alpha(p)=(2\alpha-\tfrac{2\alpha}{p})$. 
\end{prop}

Note that $\alpha(2)=\alpha$. By \cite{GrSe94} the estimate \eqref{L2Hardy} holds with $\alpha=1/4$, proving Corollary \ref{corollary} for $1<p<2$.

\bp
First, we construct an analytic family of operators. Let $\chi$ be a smooth bump function supported near the origin, and define $T_z f=A[\chi (I-\Delta)^{\frac{z}{2}} f]$. Then we prove
\begin{align}
   \label{analyticL2} \|T_{z} f\|_2&\le C \|f\|_2, \qquad \hfill \mathrm{Re} \, z=\alpha \\
    \label{analyticH1} \|T_{z} f\|_{L^1}&\le C_y \|f\|_{\mathcal{H}^1}, \qquad \hfill \mathrm{Re} \, z=0.
\end{align}
Here $\mathcal{H}^1=\mathcal{H}^1(\rr^3)$ refers to the Hardy space on $\rr^3$. The estimates \eqref{analyticL2} and \eqref{analyticH1} follow from \eqref{L2Hardy}, the fact that $A$ is bounded on $L^1$, and the fact that $(I-\Delta)^{iy/2}$ is a Calderon-Zygmund operator, and thus bounded from $L^2\to L^2$ and $\mathcal{H}^1\to L^1$ with constants depending at most polynomially on $y$. Thus we can use the analytic interpolation theorem in \cite{FeSt72} to deduce that the operator
\[
T_{\alpha(p)}: L^p(\rr^3)\to L^p(\rr^3)
\] 
is bounded for $1<p<2$. Since $A$ is not translation-invariant we cannot commute $A$ and $(I-\Delta)^{\alpha(p)/2}$. However, we can use Littlewood-Paley theory to achieve nearly the same result. For $f\in L^p$ supported in $\supp\chi$ we can write
\begin{align*}
    \|A f\|_{L^p_{\alpha(p)}}\le \Big\|\Big(\sum_{k\ge 0} |2^{k\alpha(p)}P_k A\chi f|^2\Big)^{1/2}\Big\|_{L^p}.
\end{align*}
By two applications of Lemma \ref{LittlewoodPaley} and the fact that $P_k\big(2^{-2k}(I-\Delta)\big)^{\alpha(p)/2}$ is also a Littlewood-Paley multiplier of order $k$, we can estimate
\begin{align*}
    \Big\|\Big(\sum_{k\ge 0} |2^{k\alpha(p)}P_k A \chi f|^2\Big)^{1/2}\Big\|_{p}&\le \Big\|\Big(\sum_{k\ge 0} |P_k A_k\chi P_k(I-\Delta)^{\alpha(p)/2} f|^2\Big)^{1/2}\Big\|_{p}+C\|f\|_p \\
    &\le \Big\|\Big(\sum_{k\ge 0} |P_k A \chi (I-\Delta)^{\alpha(p)/2} f|^2\Big)^{1/2}\Big\|_{p}+C\|f\|_p \\
    &\le \| T_{\alpha(p)} f\|_{p} +C\|f\|_{p} \\
    &\lesssim \|f\|_{p}
\end{align*}
finishing the proof.
\ep

\appendix

\section{Properties of Heisenberg-Sobolev Space}\label{triebelheisenbergappendix}

We prove the following Proposition.

\begin{prop}
    The definition of the Heisenberg-Sobolev norm in Definition \ref{heisensobonorm} is independent of the choice of $\psi$.
\end{prop}
Suppose $\{\tilde{\psi}_\lambda\}_{\lambda\in\Lambda}$ is another partition of unity satisfying the conditions in Definition \ref{heisensobonorm}. Observe that because the action of $\Lambda$ on $\mathbb{H}^1$ is properly discontinuous there is a finite set $\B\subset \Lambda$ contained in the Euclidean ball $B_4(0)$ (independent of ${\tilde{\psi}}$ and $\psi$) such that
\[
\psi_0=\psi_0\Big(\sum_{\sigma\in \B} {\tilde{\psi}}_\sigma\Big).
\]
Next, for each $\sigma\in \B$ and $\lambda\in \Lambda$ we have
\begin{align*}
    \psi_0{\tilde{\psi}}_\sigma \mathcal{R}_{\lambda^{-1}} f&=\psi_0 \mathcal{R}_\sigma {\tilde{\psi}}_0 \mathcal{R}_{\sigma^{-1}} \mathcal{R}_{\lambda^{-1}} f \\
    &=\mathcal{R}_\sigma \psi_{\sigma^{-1}} {\tilde{\psi}}_0 \mathcal{R}_{(\sigma\lambda)^{-1}} f.
\end{align*}
Since the supports of $\psi_\lambda$ are finitely overlapping and $\B$ is finite, we obtain
\begin{align}\label{sigmaflip}
    \Big\|\sum_{\lambda\in\Lambda} \mathcal{R}_\lambda D^s \psi_0 \mathcal{R}_{\lambda^{-1}} f\Big\|_p &=\Big\|\sum_{\lambda\in\Lambda} \mathcal{R}_\lambda D^s \sum_{\sigma\in \B} \psi_0{\tilde{\psi}}_\sigma \mathcal{R}_{\lambda^{-1}} f\Big\|_p \\
    &=\Big\|\sum_{\lambda\in\Lambda} \sum_{\sigma\in \B} \mathcal{R}_{\lambda} D^s \mathcal{R}_\sigma \psi_{\sigma^{-1}} {\tilde{\psi}}_0 \mathcal{R}_{(\sigma\lambda)^{-1}} f\Big\|_p. \notag \\
    &\simeq \Big(\sum_{\lambda\in \Lambda} \sum_{\sigma\in \B} \|\mathcal{R}_\lambda D^s \mathcal{R}_\sigma \psi_{\sigma^{-1}} {\tilde{\psi}}_0 \mathcal{R}_{(\sigma\lambda)^{-1}} f\|_p^p\Big)^\frac1p. \notag
\end{align}
Let $g_{\lambda,\sigma}=\psi_{\sigma^{-1}}{\tilde{\psi}}_0 \mathcal{R}_{(\sigma\lambda)^{-1}}f$. We prove that $\|D^s \mathcal{R}_\sigma g_{\lambda,\sigma}\|_p\simeq\|\mathcal{R}_\sigma D^s g_{\lambda,\sigma}\|_p$ uniformly in $\sigma$ and $\lambda$. To show this we need some technical details from the definition of Triebel-Lizorkin spaces (cf. \cite{Tr83,Tr92}). 

\begin{defn}
Let $\Omega$ be the collection of all sequences $\{\omega_j\}_{j=0}^\infty\subset \mathcal{S}(\rr^3)$ with the properties
\be
\item there exist positive constants $A,B,C$ such that 
\begin{align*}
    \supp\,  \omega_0&\subset \{\xi \, : \, |\xi|\le A\} \\
    \supp\,  \omega_j&\subset \{\xi \, : \, B2^{j-1}\le |\xi|\le C2^{j+1}\}, \qquad j=1,2,3,...
\end{align*}
\item for every multi-index $\alpha$ there exists $c_\alpha>0$ such that
\[
\sup_{x\in\rr^3}\sup_{j\in\nn}2^{j|\alpha|}|\partial^\alpha \omega_j(\xi)|\le c_\alpha,
\]
\item for every $\xi\in\rr^3$
\[
\sum_{j=0}^\infty \omega_j(\xi)=1.
\]
\ee
For a sequence $\{\omega_j\}\in \Omega$ we define the Triebel-Lizorkin norm
\[
\|f\|_{F^{p,q}_s}=\Big\|\Big(\sum_{j=0}^\infty|2^{js}\widecheck{\omega}_j*f|^q\Big)^{1/q} \Big\|_{L^p}.
\]
We remark that a different choice of $\{\omega_j\}$ results in an equivalent norm.
\end{defn}

Let $\{\omega_j\}\in\Omega$ with associated constants $A,B,C,c_\alpha$. Recall that $\|\mathcal{R}_\sigma g_{\lambda,\sigma}\|_{L^p_s}\simeq \|\mathcal{R}_\sigma g_{\lambda,\sigma}\|_{F^{p,2}_s}$.  
A direct calculation reveals that 
\[
    \reallywidehat{\mathcal{R}_\sigma g}(\xi)=e^{-i\langle\sigma, \xi\rangle} \widehat{g}(\xi_1+\tfrac{\sigma_2}{2}\xi_3,\xi-\tfrac{\sigma_1}{2}\xi_3,\xi_3)
\]
We define $ \vartheta(\eta)=(\eta_1-\tfrac{\sigma_2}{2}\eta_3,\eta_2+\tfrac{\sigma_1}{2}\eta_3,\eta_3).$ Then by a linear change of variables 
\begin{align*}
    \widecheck{\omega_j}* \mathcal{R}_\sigma g_{\lambda,\sigma}&=\int e^{i\langle x, \xi\rangle} \omega_j(\xi) e^{-i\langle\sigma, \xi\rangle} \widehat{g}_{\lambda,\sigma}(\eta(\xi)) \, d\xi \\
    & \int e^{i(\langle x\odot\sigma^{-1}),\eta\rangle}\omega_j( \vartheta(\eta))\widehat{g}_{\lambda,\sigma}(\eta) \, d\eta \\
    &=\mathcal{R}_{\sigma}\Big[\widecheck{\omega_j\circ  \vartheta}*g_{\lambda,\sigma}\Big].
\end{align*}

The smooth cutoff $\omega_j\circ \vartheta$, $j=1,2,3...$ is supported where 
\[
B2^{j-1}\le | \vartheta(\eta)| \le C2^{j+1}.
\]
Since $|\sigma_j|\le 4$ for all $\sigma\in\B$ these inequalities imply that 
\[
\supp\,  \omega_j( \vartheta(\eta))\subset\big\{\eta \, : \, \tfrac{B}{5} 2^{j-1}\le|\eta|\le 5C 2^{j+1}\big\}.
\]
The same argument also implies that $\supp\, \omega_0 ( \vartheta(\eta))\subset \{\eta \, : \, |\eta|\le 5A\}$. 
Next, since $ \vartheta(\eta)$ is linear and $|\sigma_j|\le 4$ for $j=1,2,3$ we can conclude that for any multi-index $\alpha$
\[
\sup_{\eta\in\rr^3}\sup_{j\in\nn} 2^{j|\alpha|}|\partial^\alpha \omega_j( \vartheta(\eta))|\le 3^{|\alpha|}c_\alpha.
\]
Since clearly $\sum_{j=0}^\infty \omega_j( \vartheta(\eta))=1$ for every $\eta$ we conclude that $\{\tilde{\omega}_j\}_{j=0}^\infty=\{\omega_j\circ \vartheta\}_{j=0}^\infty\in \Omega$, hence 
\begin{align*}
    \|D^s \mathcal{R}_\sigma g_{\lambda,\sigma}\|_p&\simeq \Big\|\Big(\sum_{j=0}^\infty |2^{js} \widecheck{\omega}_j*\mathcal{R}_\sigma g_{\lambda,\sigma} |^2\Big)^{1/2}\Big\|_p \\
    &= \Big\|\mathcal{R}_\sigma \Big(\sum_{j=0}^\infty |2^{js}\widecheck{\omega_j\circ  \vartheta}*g_{\lambda,\sigma}|^2\Big)^{1/2}\Big\|_p \simeq \|\mathcal{R}_\sigma D^s g_{\lambda,\sigma}\|_p.
\end{align*}

Plugging this into \eqref{sigmaflip} we obtain
\begin{align*}
    \Big\|\sum_{\lambda\in\Lambda} \mathcal{R}_\lambda D^s \psi_0 \mathcal{R}_{\lambda^{-1}} f\Big\|_p&\simeq C\Big(\sum_{\lambda\in \Lambda} \sum_{\sigma\in \B} \|\mathcal{R}_{\sigma\lambda} D^s \psi_{\sigma^{-1}}{\tilde{\psi}}_{0}\mathcal{R}_{(\sigma\lambda)^{-1}}f\|_p^p\Big)^\frac1p \\
    &\le C \Big(\sum_{\tilde{\lambda}\in \Lambda} \sum_{\sigma\in \B} \|\mathcal{R}_{\tilde{\lambda}} D^s \psi_{\sigma^{-1}} {\tilde{\psi}}_0 \mathcal{R}_{{\tilde{\lambda}}^{-1}} f\|_p^p\Big)^\frac1p \\
    &\simeq \Big\|\sum_{{\tilde{\lambda}}\in\Lambda} \mathcal{R}_{\tilde{\lambda}} D^s {\tilde{\psi}}_0\Big(\sum_{\sigma\in \B} \psi_{\sigma^{-1}}\Big) \mathcal{R}_{{\tilde{\lambda}}^{-1}} f\Big\|_p \\
    &=C\Big\|\sum_{{\tilde{\lambda}}\in\Lambda} \mathcal{R}_{\tilde{\lambda}} D^s {\tilde{\psi}}_0 \mathcal{R}_{{\tilde{\lambda}}^{-1}} f\Big\|_p,
\end{align*}
proving that the Heisenberg-Sobolev norm is independent of choice of cutoff function.

\bibliography{mybib}
	\bibliographystyle{plain}

\end{document}